\documentclass[11pt]{article}
\usepackage{latexsym,amsmath,amssymb}

\newif\ifdviwin

\usepackage[center]{titlesec}
\usepackage[english]{babel}
\usepackage{indentfirst}
\usepackage[mathscr]{eucal}
\usepackage{amssymb,amsmath,amsfonts}
\usepackage{fancybox,fancyhdr}
\usepackage{graphicx}
\usepackage[utf8]{inputenc}
\usepackage{float}
\usepackage{color}
\usepackage[pdftex]{hyperref}
\hypersetup{colorlinks=true,linkcolor=blue,citecolor=blue} 
\usepackage[export]{adjustbox}

\newif\ifdviwin

\dviwintrue

\def\cR{\mathcal{R}}

\def\H{\mathcal{H}}
\def\H{\H}

\def\m2r{\mathbb{M}^2\times\mathbb{R}}
\def\h2r{\mathbb{H}^2\times\mathbb{R}}

\let\infty=\infty \let\0=\emptyset

\let\tilde=\widetilde

\let\alfa=\alpha

\let\parc=\partial
\let\varepsilon=\varepsilon

\def\cte.{\mathop{\rm cte.}\nolimits}

\def\N{\mathbb{N}}

\def\R{\mathbb{R}}

\def\M{\mathbb{M}}

\def\H{\mathcal{H}}

\def\cW{\mathcal{W}}

\def\S{\mathbb{S}}

\def\m2r{\M^2\times\R}
\def\mkr{\M^2(\kappa)\times\R}

\def\h2r{\mathbb{H}^2\times\R}
\def\s2r{\mathbb{S}^2\times\R}

\def\hipar{\H\in \mathfrak{C}^1_{\kappa}}
\def\hi{\H\in C^1([-1,1])}
\def\Hs{\mathcal{H}\text{-}\mathrm{surface}}
\def\Hss{\mathcal{H}\text{-}\mathrm{surfaces}}

\def\Hgg{\mathcal{H}\text{-}\mathrm{graphs}}
\def\sig{\Sigma}
\def\r3{\mathbb{R}^3}
\def\arctanh{\mathrm{arctanh}}

\def\c1{\mathfrak{C}^1_{1,\mathrm{even}}}
\def\cm1{\mathfrak{C}^1_{-1,\mathrm{even}}}

 \newtheorem{defi}{Definition}[section]
 \newtheorem{teo}[defi]{Theorem}
 \newtheorem{pro}[defi]{Proposition}
 \newtheorem{cor}[defi]{Corollary}
 \newtheorem{lem}[defi]{Lemma}
 
 \newtheorem{remark}[defi]{Remark}
 
 \newtheorem{obs}[defi]{Observation}

 \newenvironment{proof}{\rm \trivlist \item[\hskip \labelsep{\it
      Proof}:]}{\nopagebreak \hfill $\Box$ \endtrivlist}

\numberwithin{equation}{section}

\begin{document}


\begin{center}

\renewcommand{\thefootnote}{\,}
{\large \bf  A Delaunay-type classification result for prescribed mean curvature surfaces in $\mathbb{M}^2(\kappa)\times\R$
\footnote{\hspace{-.75cm}
The author was partially supported by MICINN-FEDER Grant No. MTM2016-80313-P, Junta de Andalucía Grant No. FQM325 and FPI-MINECO Grant No. BES-2014-067663.}}\\
\vspace{0.5cm} { Antonio Bueno}\\
\end{center}
\vspace{.5cm}
Departamento de Geometría y Topología, Universidad de Granada, E-18071 Granada, Spain. \\ 
\emph{E-mail address:} jabueno@ugr.es \vspace{0.3cm}

\begin{abstract}
The purpose of this paper is to study immersed surfaces in the product spaces $\mkr$, whose mean curvature is given as a $C^1$ function depending on their angle function. This class of surfaces extends widely, among others, the well-known theory of surfaces with constant mean curvature. In this paper we give necessary and sufficient conditions for the existence of prescribed mean curvature spheres, and we describe complete surfaces of revolution proving that they behave as the Delaunay surfaces of CMC type.
\end{abstract}

\section{\large Introduction}
\vspace{-.5cm}
Let $\H$ be a $C^1$ function defined on the 2-sphere of the Euclidean space $\R^3$. An immersed, oriented surface $\sig$ in $\r3$ is said to have \emph{prescribed mean curvature} $\H$ (for short, $\sig$ is an $\H$\emph{-surface)} if its mean curvature function $H_\sig$ satisfies
\begin{equation}\label{Hsupr3}
H_\sig(p)=\H(\eta_p),\hspace{.5cm} \forall p\in\sig,
\end{equation}
where $\eta$ denotes the \emph{Gauss map} of $\sig$.  Obviously, when $\H=H_0$ is chosen as a constant, the surfaces defined by Equation \eqref{Hsupr3} are just the surfaces with constant mean curvature equal to $H_0$. 

The definition of this class of immersed surfaces is motivated by a long standing conjecture due to Alexandrov \cite{Ale} regarding the uniqueness of strictly convex spheres\footnote{By \emph{sphere} we mean a closed (compact and without boundary) surface of genus zero.} with \emph{prescribed Weingarten curvature}, i.e. in Equation \eqref{Hsupr3} the function $\H$ is an arbitrary symmetric function of its principal curvatures and its Gauss map. This conjecture has been recently solved by Gálvez and Mira as a consequence of their outstanding work \cite{GaMi1}, where the authors announced an extremely general Hopf-type theorem\footnote{In the literature, a \emph{Hopf-type theorem} refers to a uniqueness result of immersed spheres in some class of immersed surfaces} for immersed surfaces governed by an elliptic PDE in an arbitrary oriented three-manifold.  For the particular but important case when we prescribe the mean curvature, that is when $\sig$ is governed by Equation \eqref{Hsupr3}, this uniqueness result states the following: \emph{let $\mathcal{S}$ be a strictly convex sphere satisfying \eqref{Hsupr3}. Then, any other immersed sphere governed by \eqref{Hsupr3} is a translation of $\mathcal{S}$}.

Besides the work of Guan and Guan \cite{GuGu}, where they proved the existence of $\H$-spheres under symmetry conditions on $\H$, the class of immersed surfaces in $\r3$ defined by Equation \eqref{Hsupr3} was largely unexplored until we recently developed the global theory of prescribed mean curvature surfaces in $\R^n$, in joint work with Gálvez and Mira \cite{BGM1,BGM2}. In \cite{BGM1} we focused in the global theory of $\H$-surfaces in $\r3$, obtaining a priori curvature and height estimates for $\Hgg$, a structure theorem for properly embedded $\Hss$ with finite topology, stability properties and a radius estimate for stable $\Hss$. In \cite{BGM2} we covered topics such as the analysis of rotational $\H$-hypersurfaces obtaining Delaunay-type classification result, and exhibited examples of a vast variety of $\H$-hypersurfaces for general choices of the prescribed function. 

A particular but important case is when $\H\in C^1(\S^2)$ depends only on the height of the sphere. These functions are called \emph{rotationally symmetric} for obvious reasons, and thus there exists a one dimensional function $\mathfrak{h}\in C^1([-1,1])$ such that $\H(x)=\mathfrak{h}(\langle x,e_3\rangle)$, for every $x\in \S^2$. For this particular case, Equation \eqref{Hsupr3} reads as
\begin{equation}\label{presim}
H_\sig(p)=\mathfrak{h}(\langle\eta_p,e_3\rangle)=\mathfrak{h}(\nu(p)),\hspace{.5cm} \forall p\in\sig,
\end{equation}
where
\begin{equation}\label{defangulo}
\nu:\sig\rightarrow\R,\hspace{.5cm} \nu(p):=\langle\eta_p,e_3\rangle,\hspace{.5cm} \forall p\in\sig,
\end{equation}
is the so called \emph{angle function}.

Our aim in this paper is to extend this recently developed theory of $\Hss$ to the product spaces $\mkr$, where $\M^2(\kappa)$ stands for the complete, simply connected surface of constant curvature $\kappa$. We take as a starting point the natural mixture between the well-studied theory of constant mean curvature surfaces immersed in these product spaces and the theory of $\Hss$ in $\r3$.

The theory of immersed surfaces in $\mkr$ has experimented an extraordinary development since Abresch and Rosenberg \cite{AbRo} defined a \emph{holomorphic quadratic differential} on every constant mean curvature surface, that vanishes on rotational examples. This quadratic differential, called in the literature the \emph{Abresch-Rosenberg differential}, enabled the authors to extend the so called Hopf theorem: an immersion of a constant mean curvature sphere in $\mkr$ is a rotational, embedded sphere. This milestone was the starting point for the growth of a fruitful theory of positive, constant mean curvature surfaces (\emph{CMC surfaces} in the following) in $\mkr$; see \cite{AbRo,HLR,NeRo} for a global picture of the development of this theory.

When trying to extend Equation \eqref{Hsupr3} to the product spaces $\mkr$, we find out two major difficulties: the spaces $\S^2(\kappa)\times\R$ do not carry a Lie group structure, and thus we are not able to define a left-invariant Gauss map in order to prescribe some function on a fixed sphere, just like in Equation \eqref{Hsupr3}. The difficulty when trying to extend this theory to the spaces $\mathbb{H}^2(\kappa)\times\R$ comes from the fact that they have two non-isometric Lie group structures: one unimodular and other non-unimodular, see \cite{MePe} for details.

Nonetheless, in $\mkr$ we have a notion of angle function as well, defined by measuring the projection of a unit normal vector field $\eta$ of $\sig$ onto the vertical Killing vector field $\partial_z$, just like in Equation \eqref{defangulo}. Bearing this in mind, we can define the following class of immersed surfaces in $\mkr$:
\begin{defi}\label{defiHsup}
Let be $\H\in C^1([-1,1])$. An immersed, oriented surface $\sig$ in $\M^2(\kappa)\times\R$ has \emph{prescribed mean curvature} $\H$ if its mean curvature $H_\sig$ satisfies
\begin{equation}\label{defHsup}
H_\sig(p)=\H(\nu(p)),\hspace{.5cm} \forall p\in\sig,
\end{equation}
where $\nu(p):=\langle\eta_p,\partial_z\rangle$ is the \emph{angle function} of $\sig$ and $\eta$ is a unit normal vector field on $\sig$.
\end{defi}
In analogy with the Euclidean case, we will simply say that $\sig$ is an $\H$\emph{-surface}.

Besides the trivial choice of $\H$ as a constant, there are other prescribed functions that generate some known classes of immersed surfaces in $\mkr$. Indeed, if we consider the function $\H(x)=x,\ \forall x\in [-1,1]$, the class of immersed surfaces described by Equation \eqref{defHsup} are the self-translating solitons of the mean curvature flow (MCF for short); see \cite{Bue,LiMa} for a recent development of this theory. These particular, almost trivial, choices of the prescribed function $\H$ and the classes of immersed surfaces generated by them, show the richness and wideness of the family of $\Hss$.

The rest of the introduction is devoted to detail the organization of the paper, and highlight some of the main results.

In \textbf{Section \ref{sec:2}} we will exhibit some basic properties of immersed $\Hss$ in the product spaces. We will show that $\Hss$ obey the geometric maximum principle, as they are locally governed by an elliptic, second order, quasilinear PDE. We make special emphasis in the ambient isometries that preserve the class of $\Hss$; any such isometry must keep invariant the angle function in order to preserve Equation \eqref{defHsup}. Specifically, in Lemma \ref{sime} we will show that, for an arbitrary prescribed function $\H$, almost all the isometries of the spaces $\mkr$ are induced as isometries for the class of $\Hss$.

When studying the properties of CMC surfaces in the product spaces, one of the key tools is the existence of a sphere with the same mean curvature. This is the motivation for the contents of Section \ref{analisisfases}, where we take advantage of the symmetries of Equation \eqref{defHsup} to study rotational $\Hss$ immersed in $\mkr$. We should emphasize that the arbitrariness of the prescribed function $\H$ prevents us from obtaining a \emph{first integral} to study these rotational examples. 

In the same fashion as in Section 2 in \cite{BGM2}, we approach this study by means of a phase plane analysis of the resulting ODE that the profile curve of a rotational $\Hs$ satisfies. In Section \ref{necesariaesfera} we obtain some necessary conditions on $\H$ in order to ensure the existence of an $\H$-sphere. We also prove in Theorem \ref{teoes} that a rotational $\H$-sphere has monotonous angle function, and thus is unique among all immersed $\H$-spheres due to Gálvez and Mira uniqueness theorem \cite{GaMi1}. 

Finally, in \textbf{Section \ref{clasification}} we give a classification of complete, rotational $\Hss$, provided that $\hipar$; see Equation \eqref{espaciosc1h} for a definition of the space $\mathfrak{C}^1_\kappa$. In Theorem \ref{existenciaesfera} we prove the existence of a rotational, embedded $\H$-sphere, and in Theorem \ref{delaunay} we announce a Delaunay-type classification result for $\Hss$. In the same fashion as in the CMC case, every complete rotational $\Hs$ is either an $\H$-sphere, a vertical circular cylinder,  a properly embedded surface of \emph{unduloid type} or a properly immersed (non-embedded) surface of \emph{nodoid type}. Moreover, in analogy to the CMC case, in the space $\s2r$ there also exist rotational, embedded $\H$-surfaces diffeomorphic to $\S^1\times\S^1$, i.e. rotational, embedded \emph{$\H$-tori}.
%
%

\section{\large Basic properties of $\H$-surfaces in the product spaces}\label{sec:2}
\vspace{-.5cm}

Let $\M^2(\kappa)$ be the complete, simply connected surface with constant curvature $\kappa$. Then, up to isometries, $\M^2(\kappa)$ is one of the following surfaces: if $\kappa=0$ we get the usual flat plane $\R^2$; if $\kappa<0$ we obtain the hyperbolic plane $\mathbb{H}^2(1/\sqrt{-\kappa})$; and if $\kappa>0$ we have the 2-sphere $\S^2(1/\sqrt{\kappa})$. We drop out the case $\kappa=0$, since the theory of immersed $\Hss$ in $\R^3$ has been widely studied in \cite{BGM1,BGM2}. When $\kappa\neq 0$, these non-flat surfaces can be regarded as isometric immersions in the space $\R^3_\kappa$, where $\R^3_\kappa$ stands for the usual Euclidean space if $\kappa>0$, or for the Lorentz-Minkowski space $\mathbb{L}^3$ (that is, $\r3$ endowed with the metric with signature $+,+,-$) if $\kappa<0$. Indeed, the surface $\M^2(\kappa)$ can be defined as the quadric
\begin{equation}\label{modelobase}
\M^2(\kappa)=\{(x_1,x_2,x_3)\in\R^3_\kappa;\ x_1^2+x_2^2+\kappa x_3^2=1/\kappa,\ (1-\kappa)x_3\geq 0\}.
\end{equation}

Up to an homothetic change of the metric, we will suppose that $\kappa=\pm 1$. Henceforth, we will drop the dependence on $\kappa$ and just write $\M^2$.

The product spaces $\m2r$ are defined as the riemannian product of the surface $\M^2$ with the real line, endowed with the usual product metric. In $\m2r$ we have the two usual projections $\pi_1:\m2r\rightarrow\M^2$ and $\pi_2:\m2r\rightarrow\R$. The \emph{height function} is defined to be the second projection, and is commonly denoted by $z(p):=\pi_2(p)$ for all $p\in\m2r$. The gradient of the height function is a unitary Killing vector field, which is usually denoted in the literature by $\partial_z$, and the direction generated by $\partial_z$ is called the \emph{vertical direction}. The projection $\pi_1$ is a riemannian submersion, whose fibers are the \emph{vertical lines} $\{p\}\times\R:=\pi_1^{-1}(p)$, for $p\in\M^2$. The \emph{vertical planes} are the surfaces given by $\gamma\times\R:=\pi_1^{-1}(\gamma)$, where $\gamma\subset\M^2$  is a geodesic, which are totally geodesic surfaces isometric to $\R^2$. The \emph{horizontal planes} are the surfaces given by $\M^2\times\{t_0\}:=\pi_2^{-1}(t_0),\ t_0\in\R$, which are totally geodesic surfaces isometric to $\M^2$.

Let $\sig$ be an $\Hs$ and $\eta$ a unit normal vector field on $\sig$, and take some $p\in\sig$. Suppose that $\eta_p$ is not an horizontal vector. Thus, the implicit function theorem ensures us that a neighborhood of $p$ in $\sig$ can be expressed as a vertical graph $(x,u(x))$ of a function $u:\Omega\subset\M^2\rightarrow\R$. In this situation, Equation \eqref{defHsup} has the following divergence expression
$$
{\rm div_{\M^2}}\left(\displaystyle \frac{\nabla^{\M^2} u}{\sqrt{1+|\nabla^{\M^2} u|^2}}\right) = 2\H \left(\displaystyle\frac{1}{\sqrt{1+|\nabla^{\M^2}u|^2}}\right),
$$
where ${\rm div_{\M^2}}$ and $\nabla^{\M^2}$ are the divergence and gradient operators, both computed w.r.t. the metric of $\M^2$. If $\eta_p$ is horizontal, then $\sig$ can be expressed as a horizontal graph which also satisfies a divergence-type equation, see e.g. Section 5 in \cite{Maz}. In particular, $\Hss$ satisfy the Hopf maximum principle in its interior and boundary versions, a property that has the following geometric implication:

\begin{lem}[\textbf{Maximum principle for $\H$-surfaces}]\label{pmax}
Given $\hi$, let $\Sigma_1,\Sigma_2$ be two $\H$-surfaces, possibly with non-empty, smooth boundary. Assume that one of the following two conditions holds:
\begin{enumerate}
\item
There exists $p\in {\rm int}(\Sigma_1)\cap {\rm int}(\Sigma_2)$ such that $\eta_1(p)=\eta_2(p)$, where $\eta_i$ is the unit normal of $\Sigma_i$, $i=1,2$.
\item
There exists $p\in \partial \Sigma_1 \cap \partial \Sigma_2$ such that $\eta_1(p)=\eta_2(p)$ and $\xi_1(p)=\xi_2(p)$, where $\xi_i$ denotes the interior unit conormal of $\partial \Sigma_i$.
\end{enumerate}
Assume moreover that $\Sigma_1$ lies around $p$ at one side of $\Sigma_2$. Then $\Sigma_1=\Sigma_2$.
\end{lem}

Now let us focus in the ambient isometries and how they act on the class of $\Hss$. Besides the space forms $\R^3,\mathbb{H}^3$ and $\S^3$, whose isometry group has dimension six (the highest for a 3-dimensional space), the product spaces $\m2r$ have isometry group of dimension four, the second highest for a 3-dimensional space. Indeed, the product structure decomposes the isometry group $\mathrm{Iso}(\m2r)$ as $\mathrm{Iso}(\M^2)\times\mathrm{Iso}(\R)$. Notice that $\mathrm{Iso}(\R)$ is just the group of translations in the real line, and their elements are described as $T_\lambda:\R\rightarrow\R,\ T_\lambda(a):=a+\lambda$, for every $a\in\R$. Thus, every isometry $\Phi\in \mathrm{Iso}(\m2r)$ decomposes as $\Phi=\Phi_{\M^2}\times T_\lambda$, for some $\lambda\in\R$. The isometry $\mathrm{Id}_{\M^2}\times T_\lambda$ is commonly known as the \emph{vertical translation of (oriented) distance $\lambda$}. The 1-parameter group $\lambda\mapsto\mathrm{Id}_{\M^2}\times T_\lambda$ is the flow of the vertical Killing vector field $\partial_z$.

Given a point $p\in\M^2$, consider the rotation $\mathrm{Rot}_{\M^2,p}$ of $\M^2$ that fixes $p$. Then, the isometry $\mathrm{Rot}_{\M^2,p}\times\mathrm{Id}_\R$ is a rotation in $\m2r$ that leaves pointwise fixed the vertical line $\{p\}\times\R$. In the same fashion, let $\mathrm{R}_{\M^2,\gamma}$ be the reflection w.r.t. a geodesic $\gamma$ of $\M^2$. Then, the isometry $\mathrm{R}_{\M^2,\gamma}\times\mathrm{Id}_\R$ is a \emph{vertical reflection} in $\m2r$ w.r.t. the vertical plane $\gamma\times\R$. Lastly, given $a\in\R$, the isometry in $\m2r$ defined by $\cR_a(p,t):=(p,2a-t)$ for all $(p,t)\in\m2r$ is the \emph{horizontal reflection} w.r.t. the horizontal plane $\M^2\times\{a\}$.

Observe that given $\hi$, an $\Hs$ $\sig$ and an isometry $\Phi$ of $\m2r$, if we ask $\Phi(\sig)$ to be an $\Hs$ for the same prescribed function $\H$, then the r.h.s. of Equation \eqref{defHsup} implies that $\Phi$ must keep invariant the angle function $\nu$ of $\sig$. Thus, every isometry of the form $\Phi_{\M^2}\times\mathrm{Id}_\R$ will be also an isometry for the class of $\Hss$. Note that, in particular, reflections w.r.t. vertical planes and rotations are isometries for $\Hss$. Also, vertical translations $\mathrm{Id}_{\M^2}\times T_\lambda$ will be included among isometries for $\Hss$. The only missing isometries are the reflections with respect to horizontal planes, since these isometries change the value of the angle function of $\sig$ and thus the r.h.s. of Equation \eqref{defHsup}. We summarize these facts in the following lemma:

\begin{lem}\label{sime}
Given $\hi$, let be $\Sigma$ an $\H$-surface, $\eta$ a unit normal vector field of $\sig$ and $\Phi$ an isometry of $\m2r$. Moreover, suppose that $\Phi$ is not a horizontal reflection. Then, $\Phi(\Sigma)$ is an $\H$-surface in $\m2r$ with respect to the orientation given by $d\Phi(\eta)$.
\end{lem}

Suppose now that $\H$ is even, i.e. $\H(y)=\H(-y),\ \forall y\in [-1,1]$. If $\cR_a$ is a horizontal reflection for some $a\in\R$, and we take some $p\in\Sigma$, then the angle function $\nu^*$ of the reflected surface $\Sigma^*=\cR_a(\Sigma)$ in $p^*=\cR(p)$ satisfies $\nu^*(p^*)=-\nu(p)$. Hence,
$$
H_{\Sigma^*}(p^*)=\H(\nu^*(p^*))=\H(-\nu(p))=\H(\nu(p))=H_\Sigma(p).
$$
In particular, for even functions, the horizontal reflections in $\m2r$ are induced as isometries for the class of $\Hss$, and thus \emph{all the isometries of the space} $\m2r$ \emph{are also isometries for the class of immersed} $\H$\emph{-surfaces}.

The following proposition is a consequence of Lemmas \ref{pmax} and \ref{sime}, and is a generalization of Alexandrov's theorem for closed, embedded CMC surfaces.

\begin{pro}\label{alexandrov}
Let be $\hi$ and $\sig$ a closed, embedded $\H$-surface in $\h2r$ or $\S^2_+\times\R$, where $\S^2_+$ denotes an open hemisphere of $\S^2$. Then, $\sig$ is topologically a sphere and it is rotational around some vertical axis. Moreover, if $\H$ is also even, then $\sig$ is a symmetric, vertical bi-graph over some horizontal plane $\M^2\times\{t_0\},\ t_0\in\R$.
\end{pro}

\begin{proof}
If the base of the space is $\mathbb{H}^2$, then the classical Alexandrov reflection technique in $\R^3$ carries over verbatim to $\Hss$ in $\h2r$. For this, consider a geodesic $\gamma\subset\mathbb{H}^2$, and let $\mathcal{F}_\gamma$ be a foliation of $\mathbb{H}^2$ by geodesics parallel to $\gamma$. Then, we apply the classical Alexandrov reflection technique with respect to the family of vertical planes $\mathcal{F}_\gamma\times\R$, which is a foliation of $\h2r$ by totally geodesic surfaces, in order to ensure that $\sig$ is symmetric with respect to some vertical plane. By changing the geodesic $\gamma$, the result holds.

However, when the base is $\S^2$, we need $\sig$ to project onto some hemisphere $\S^2_+$, as happens for CMC surfaces\footnote{Indeed, in $\s2r$ there exist rotational CMC tori described by Pedrosa and Ritoré.}. Indeed, suppose that $\pi_1(\sig)$ is contained in some hemisphere $\S^2_+$, whose boundary is a geodesic $\Gamma\subset\S^2$. Fix some $p\in\Gamma$, consider $\mathrm{Rot}_{p,\theta}$ the rotation of angle $\theta$ in $\S^2$ that fixes $p$, and define $\Gamma_\theta:=\mathrm{Rot}_{p,\theta}(\Gamma)$. Note that $\Gamma_0=\Gamma$ and $\Gamma_0\times\R$ does not intersect $\sig$. Now, we apply Alexandrov reflection technique w.r.t. the 1-parameter family of vertical planes $\{\Gamma_\theta\times\R\}_\theta$, which yields that $\sig$ is symmetric with respect some $\Gamma_{\theta_0}\times\R$. Varying the point $p\in\gamma$, we conclude the result.

Finally, if $\H$ is even we can make reflections with respect to the foliation of horizontal planes $\M^2\times\{t\}$, $t\in\R$, and apply again Alexandrov reflection technique in order to ensure that $\sig$ is also a symmetric, vertical bi-graph, concluding the proof.
\end{proof}

\textbf{In the development of this paper, the (possible) zeros of a prescribed function $\H\in C^1([-1,1])$ will be supposed to be \emph{isolated and of finite multiplicity}.}

\section{\large A phase plane analysis}\label{analisisfases}
\vspace{-.5cm}

In the development of this section, we regard $\m2r$ as a submanifold of $\R^3_\kappa\times\R$ endowed with the metric $+,+,\mathrm{sign}(\kappa),+$, $\kappa=\pm 1$. Consider an arc-length parametrization of a regular curve $\alpha_\kappa(s)=(x_1(s),0,x_3(s),z(s))\subset\m2r$, $x_1(s)>0,\ s\in I\subset\R$, which is contained in a vertical plane passing through the point $(0,0,1,0)$, and rotate it around the vertical axis $\{(0,0,1)\}\times\R$. The orbit of $\alpha_\kappa(s)$ under this 1-parameter group of rotations generates an immersed surface $\sig$. Because $\alpha_\kappa(s)\in\m2r$, its first coordinates satisfy $x_1^2(s)+\kappa x_3^2(s)=\kappa$ and thus there exists a $C^1$ function $x(s)>0$ such that 
\begin{equation}\label{eccurva}
\alpha_\kappa(s)=(\sin_\kappa (x(s)),0,\cos_\kappa (x(s)),z(s)),\hspace{.5cm} \kappa=\pm 1,
\end{equation}
where the trigonometric function $\sin_\kappa$ is the usual sine function if $\kappa=1$ and the hyperbolic sine if $\kappa=-1$; the same holds for $\cos_\kappa$. \textbf{For saving notation, we will simply denote by $\alpha_\kappa=(x(s),z(s))$ to the profile curve defined in Equation \eqref{eccurva}.}

This construction generates a regular surface $\sig$ immersed in $\m2r$, parametrized by
\begin{equation}\label{parametrot}
\psi_\kappa(s,\theta)=(\sin_\kappa (x(s))\cos\theta,\sin_\kappa (x(s))\sin\theta,\cos_\kappa (x(s)),z(s))\hspace{.5cm} s\in I,\ \theta\in (0,2\pi).
\end{equation}
The angle function of $\sig$ at each point $\psi_\kappa(s,\theta)$ is given by $\nu(\psi_\kappa(s,\theta))=x'(s),\ \forall s\in I$.

If the same fashion, we define the function
$$
\cot_\kappa x(s)=\left\lbrace\begin{array}{ll}\cot x(s), & \mathrm{if}\ \kappa=1,\\ 
\coth x(s), & \mathrm{if}\ \kappa=-1.
\end{array}\right.
$$
For saving notation, we will omit from now the dependence of the variable $s$. With this parametrization, the principal curvatures of $\sig$ are given by
\begin{equation}\label{curvaprincipales}
\kappa_1=k_{\alpha_\kappa}=x'z''-x''z',\hspace{.5cm} \kappa_2=z'\frac{\cos_\kappa x}{\sin_ \kappa x}=z'\cot_\kappa x,
\end{equation}
where $k_{\alpha_\kappa}$ is the \emph{geodesic curvature} of $\alpha_\kappa$.

In this setting, the mean curvature $H_\sig$ of $\sig$ satisfies the ODE
\begin{equation}\label{ODE1}
2H_\sig(\psi_\kappa)=x'z''-x''z'+z'\cot_\kappa x.
\end{equation}
As $\alpha_\kappa$ is an arc-length parametrized curve, the relation $x'^2+z'^2=1$ holds, and thus the function $x$ is a solution to the second order autonomous ODE
\begin{equation}\label{ODE2}
x''=\frac{1-x'^2}{\tan_\kappa x}-2\varepsilon H_\sig\sqrt{1-x'^2},\hspace{.5cm} \varepsilon=\mathrm{sign}(z'),
\end{equation}
on every subinterval $J\subset I$ where $z'(s)\neq 0,\ \forall s\in J$.

Suppose now that $\sig$ is an $\Hs$ for some $\hi$, i.e. $H_\sig(\psi_\kappa)=\H(x')$. If we denote by $y=z'$, we can write Equation \eqref{ODE2} as the first order autonomous system
\begin{equation}\label{1ordersys}
\left(\begin{array}{c}
x'\\
y'
\end{array}\right)=\left(\begin{array}{c}
y\\
\displaystyle{\frac{1-y^2}{\tan_\kappa x}}-2\varepsilon\H(y)\sqrt{1-y^2}
\end{array}\right)=F_{\kappa,\varepsilon}(x,y).
\end{equation}
At this point, we shall study system \eqref{1ordersys} with a \emph{phase plane} analysis as the authors did in \cite{BGM2}.

We define the phase plane $\Theta^\kappa_\varepsilon$ of \eqref{1ordersys} as the half-strip $\Theta^{-1}_\varepsilon:=(0,\infty)\times(-1,1)$ if $\kappa=-1$ and $\Theta^1_\varepsilon:=(0,\pi)\times(-1,1)$ if $\kappa=1$, with coordinates $(x,y)$ denoting, respectively, the distance to the rotation axis and the angle function of $\alpha_\kappa$. The solutions of system \eqref{1ordersys} will be called \emph{orbits}, and will be represented by $\gamma(s)=(x(s),y(s))$. Note that in the case $\kappa=1$, as the base $\S^2$ is compact, the maximum distance that a point can reach from the axis of rotation is exactly half of the length of a great circle of $\S^2$; for that maximum distance equal to $\pi$, we already reach the antipodal axis. 

The points in $\Theta^\kappa_\varepsilon$ where $y'=0$ correspond to the points $\alpha_\kappa$ where the angle function of $\sig$ has vanishing derivative, and by Equation \eqref{curvaprincipales} these points correspond also to the points where the geodesic curvature $k_{\alpha_\kappa}$ vanishes. By analyzing the second component of the function $F_{\kappa,\varepsilon}(x,y)$ in \eqref{1ordersys}, we conclude that these points are placed at the intersection of $\Theta^\kappa_\varepsilon$ with the (possibly disconnected) horizontal graph given by
\begin{equation}\label{graga}
x=\Gamma^\kappa_\varepsilon(y)=\arctan_\kappa\left(\frac{\sqrt{1-y^2}}{2\varepsilon\H(y)}\right).
\end{equation}
We will denote by $\Gamma^\kappa_\varepsilon=\Theta^\kappa_\varepsilon\cap\{x=\Gamma^\kappa_\varepsilon(y)\}$. Note that in some cases the curve $\Gamma^\kappa_\varepsilon$ might be empty; for example, for the case $\kappa=-1$, $\H<0$ and $\varepsilon=1$.

If $\kappa=-1$ the curve $\Gamma_\varepsilon^{-1}$ has an asymptote where the $\arctanh$ function fails to be defined. This occurs when the argument of the $\arctanh$ function is equal to $\pm 1$, since for these values $\arctanh(\pm 1)=\pm\infty$. Thus, $\Gamma_\varepsilon^{-1}$ has an asymptote at the line $\{y=y_0\}$ if and only if $\sqrt{1-y_0^2}=2\varepsilon\H(y_0)$ for some $y_0\in[-1,1]$; see Figure \ref{fasesh2r}.

\begin{figure}[H]
\centering
\includegraphics[width=0.6\textwidth]{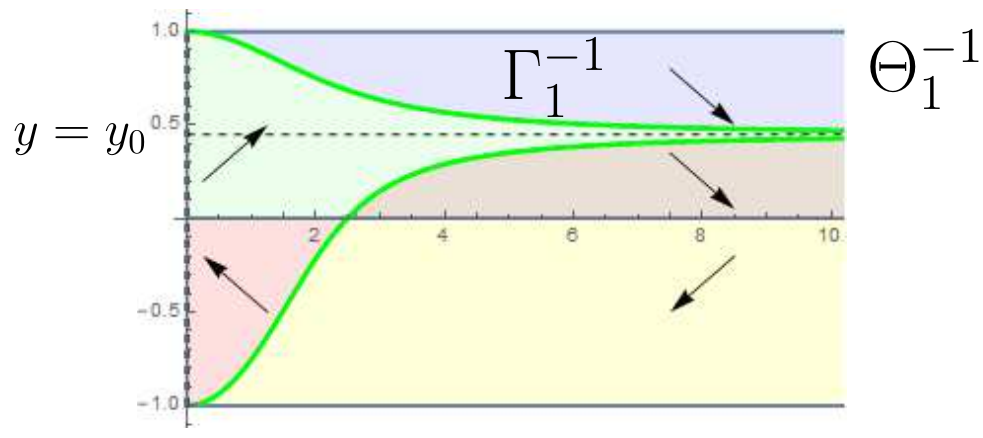}  
\caption{The phase plane $\Theta_1^{-1}$, for $\varepsilon=1$ and $\kappa=-1$. Here the curve $\Gamma_1^{-1}$ has an asymptote at some $y_0$. The arrows show how an orbit behave in each component.}
\label{fasesh2r}
\end{figure}

The case $\kappa=1$ is detailed next. Suppose that there exists some $y_0\in[-1,1]$ such that $\H(y_0)=0$. Then, $\Gamma_\varepsilon^1(y_0)=\arctan(\pm\infty)=\pm\pi/2$, proving that $\Gamma_\varepsilon^1(y)$ takes finite values at the zeros of $\H$. As we can extend by periodicity the $\arctan$ function, the graph $\Gamma_\varepsilon^1$ can be extended at the zeros\footnote{Recall that the zeros of $\H$ are finite, hence isolated.} of $\H$ as follows:
\begin{equation}\label{extensions2r}
\begin{array}{ll}
\tilde{\Gamma_\varepsilon^1}(y)=\pi+\Gamma_\varepsilon^1(y),&\ \mathrm{if}\ \H\ \mathrm{changes\ of\ sign\ around}\ y_0,\\
\tilde{\Gamma_\varepsilon^1}(y)=\Gamma_\varepsilon^1(y),&\ \mathrm{if}\ \H\ \mathrm{does\ not\ change\ of\ sign\ around}\ y_0.
\end{array}
\end{equation}
We will keep naming $\Gamma_\varepsilon^1$ to all the extensions glued at the zeros of $\H$, see Figure \ref{fasess2r}, left.

\begin{figure}[H]
\centering
\includegraphics[width=1.05\textwidth]{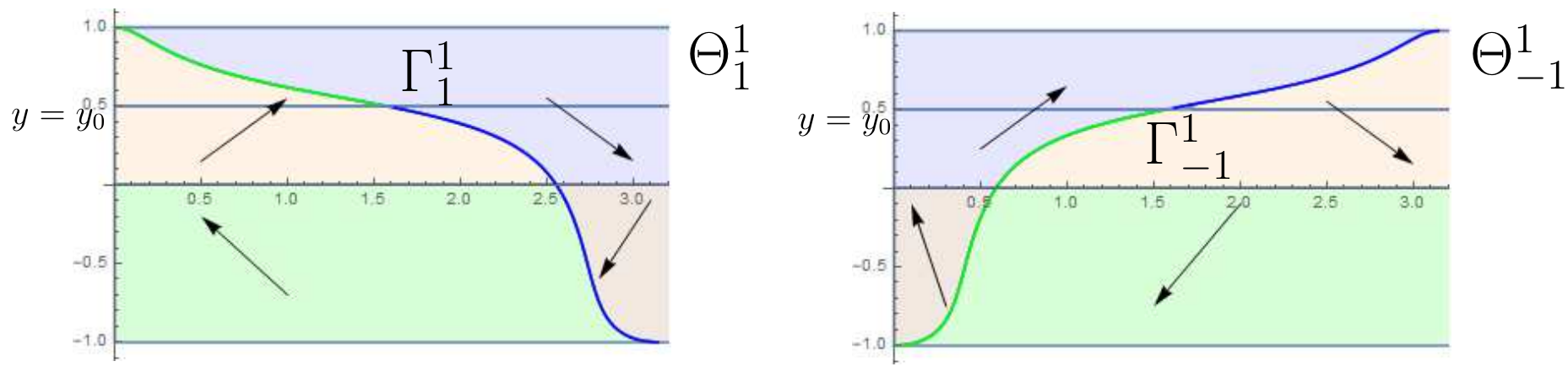}  
\caption{Left: the phase plane $\Theta_1^1$. The curve $\Gamma_1^1$ has been extended at $y_0$, where $\H$ changes of sign. The components of this extension are plotted in green and blue. Right: the phase plane $\Theta_{-1}^1$. Observe the symmetry between the phase planes $\Theta_\varepsilon^1$, for $\varepsilon=\pm 1$, w.r.t. the vertical line $\{x=\pi/2\}$.}
\label{fasess2r}
\end{figure}

Let us study deeper the case $\kappa=1$. For that, let $(x(s),y(s))$ be an orbit in $\Theta_\varepsilon^1$ and let us define $(\overline{x}(s),\overline{y}(s))=(\pi-x(-s),y(-s))$, that is $(\overline{x}(s),\overline{y}(s))$ is just the orbit $(x(s),y(s))$ symmetrized w.r.t. the vertical segment $\{x=\pi/2\}$ and with backwards movement. Then, $(\overline{x}(s),\overline{y}(s))$ is a solution of \eqref{1ordersys} for $-\varepsilon$. In particular, \textbf{the phase spaces $\Theta_\varepsilon^1$ are symmetric with respect to the vertical segment $\{x=\pi/2\}$}. This has the following implication: let $\gamma$ be an orbit in $\Theta_\varepsilon^1$ and consider its symmetric $\overline{\gamma}$ in $\Theta_{-\varepsilon}^1$. Name $\alpha_\gamma$ and $\alpha_{\overline{\gamma}}$ to the profile curves associated to $\gamma$ and $\overline{\gamma}$, respectively. Then, $\alpha_\gamma$ and $\alpha_{\overline{\gamma}}$ are symmetric w.r.t. the plane $\{(x,y,0)\}\times\R,\ (x,y,0)\in\S^2$.

This symmetry condition will play a crucial role in the study of rotational $\Hss$ in $\s2r$. For example, the graphs $\Gamma_\varepsilon^1,\ \varepsilon=\pm 1$ can be defined one by means of the other as
\begin{equation}\label{gragaenfunciondelaotra}
\Gamma_{-\varepsilon}^1(y)=\pi-\Gamma_\varepsilon^1(y),\ \forall y\in [-1,1].
\end{equation}
See Figure \ref{fasess2r}, right.

The \emph{equilibrium} points are the points $e_0^{\varepsilon,\kappa}=(x_0,y_0)\in \Theta^\kappa_\varepsilon$ such that $F_{\kappa,\varepsilon}(x_0,y_0)=0$. Note that these points must lie in the axis $\{y=0\}$, according to Equation \eqref{1ordersys}. \textbf{Henceforth, we will identify} $e_0^{\varepsilon,\kappa}\equiv(e_0^{\varepsilon,\kappa},0)$.

If $\kappa=-1$, there is a unique equilibrium $e_0^{\varepsilon,-1}$ in $\Theta^{-1}_\varepsilon$ if  $\varepsilon\H(0)>0$, namely
\begin{equation}\label{equilibrium}
e_0^{\varepsilon,-1}=\displaystyle{\arctanh\left(\frac{1}{2\varepsilon\H(0)}\right)}
\end{equation}
This equilibrium point corresponds to the case where $\sig$ has constant distance to the axis of rotation and vanishing angle function everywhere; that is, $\sig$ is a right circular cylinder $\S^1(e_0^{\varepsilon,-1})\times\R$ of constant mean curvature $\varepsilon\H(0)$ and vertical rulings. 

However, if $\kappa=1$, there are two equilibrium points $e_0^{\varepsilon,1}$, each one in $\Theta^1_\varepsilon$. These points also correspond to vertical cylinders, having distance to the axis of rotation $\pi$-complementary, i.e. $e_0^{1,1}+e_0^{-1,1}=\pi$. The equilibrium points are given by
\begin{equation}\label{equis2r}
\begin{array}{lll}
\vspace*{.4cm} e_0^{1,1}=\arctan\left(\displaystyle{\frac{1}{2\H(0)}}\right), &e_0^{-1,1}=\pi-e_0^{1,1}, & \mathrm{if}\ \arctan\left(\displaystyle{\frac{1}{2\H(0)}}\right)\geq 0\\
e_0^{1,1}=\pi+\arctan\left(\displaystyle{\frac{1}{2\H(0)}}\right), &e_0^{-1,1}=\pi+e_0^{1,1}, & \mathrm{if}\ \arctan\left(\displaystyle{\frac{1}{2\H(0)}}\right)\leq 0
\end{array}
\end{equation} 
Notice that $e_0^{1,1}=e_0^{-1,1}$ if and only if $\H(0)=0$.

Two distinct orbits cannot intersect in the phase plane, since it would be a contradiction with the uniqueness of the Cauchy problem. As a consequence, the set of all the possible orbits provide a foliation by \textbf{\emph{regular proper $C^1$ curves}} of $\Theta^\kappa_\varepsilon$ (or $\Theta^\kappa_\varepsilon-e_0^{\varepsilon,\kappa}$ if some $e_0^{\varepsilon,\kappa}$ exists). This properness condition will be applied throughout this paper, and should be interpreted as follows: an orbit $\gamma(s)$ cannot have as endpoint a finite point $(x_0,y_0)\neq e_0^{\varepsilon,\kappa}$ with $x_0\neq 0$ and $y_0\neq\pm 1$, since at these points Equation \eqref{1ordersys} has local existence and uniqueness, and thus any orbit around a point $(x_0,y_0)\in\Theta_\varepsilon^\kappa$ can be extended.

The curve $\Gamma^\kappa_\varepsilon$ and the axis $\{y=0\}$ divide $\Theta^\kappa_\varepsilon$ into connected components, having in common that the coordinates $x(s)$ and $y(s)$ of every orbit are monotonous in each component. In particular, at each of these \emph{monotonicity regions}, the geodesic curvature $k_{\alpha_\kappa}$ of $\alpha_\kappa(s)$ has constant sign. Specifically, by \eqref{curvaprincipales} we have at each $\alpha_\kappa(s)$, $\in J$:
\begin{equation}\label{signk}
\mathrm{sign}(\kappa_1)=\mathrm{sign}(-\varepsilon y'),\ \ \ \ \mathrm{sign}(\kappa_2)=\varepsilon.
\end{equation}

We can view the orbits of system \eqref{1ordersys} locally as graphs $y=y(x)$, wherever  $y\neq 0$. Specifically, we have
\begin{equation}\label{yfuncx}
y\frac{dy}{dx}=\frac{1-y^2}{\tan_\kappa x}-2\varepsilon\H(y)\sqrt{1-y^2}.
\end{equation}
Thus, in each monotonicity region the sign of the quantity $yy'$ is constant. This implies that the signs of $y_0$ and $x_0-\Gamma^\kappa_\varepsilon(y_0)$ (when $\Gamma^\kappa_\varepsilon(y_0)$ exists) determine how the orbit of \eqref{1ordersys} behaves at the point $(x_0,y_0)$. The following lemma summarizes the motion of an orbit $\gamma(s)$ in each monotonicity region. In Figure \ref{fasesh2r} we can see the monotonicity regions in a phase plane, with the behavior of an orbit in each region.
\begin{lem}\label{monotonia}
In the conditions exposed above, the different behaviors in each monotonicity region are described as follows
\begin{enumerate}
\item If $x_0>\Gamma^\kappa_{\varepsilon}(y_0)$ (resp. $x_0<\Gamma^\kappa_{\varepsilon}(y_0)$) and $y_0>0$, then $y(x)$ is strictly decreasing (resp. increasing) at $x_0$. 
\item If $x_0>\Gamma^\kappa_{\varepsilon}(y_0)$ (resp. $x_0<\Gamma^\kappa_{\varepsilon}(y_0)$) and $y_0<0$, then $y(x)$ is strictly increasing (resp. decreasing) at $x_0$.
\item If $y_0=0$, then the orbit passing through $(x_0,0)$ is orthogonal to the $x$ axis.
\item If $x_0=\Gamma^\kappa_{\varepsilon}(y_0)$, then $y'(x_0)=0$ and $y(x)$ has a local extremum at $x_0$.
\end{enumerate}
\end{lem}

The following proposition restricts the possible endpoints of an orbit. 

\begin{pro}
No orbit in $\Theta_\varepsilon^\kappa$ can converge at some point	of the form $(0,y)$ with $|y|<1$.
\end{pro}

\begin{proof}\label{orbitapuntolimite}
Arguing by contradiction, assume that $\gamma^\kappa$ is an orbit in $\Theta_\varepsilon^\kappa$ having a limit point of the form $(0,y)$, $|y|<1$, and let $\alfa_\kappa(s)=(x(s),z(s))$ denote the profile curve of its corresponding rotational $\H$-surface $\sig$. Then, $(x(s_n),x'(s_n))\to (0,y)$ for a sequence of values $s_n$, and in particular $\alfa_\kappa(s)$ approaches the rotation axis in a non-orthogonal way (since $|y|\neq 1$). So, by the monotonicity properties of the phase plane, we see that a piece of $\Sigma$ is a graph $z=u(x_1,x_2)$ in $\m2r$ defined on a punctured ball $\Omega-{\bf \{0\}}$ contained in $\M^2$. Moreover, the mean curvature function of $\Sigma$, viewed as a function $H(x_1,x_2)$ on $\Omega-{\bf \{0\}}$, extends continuously to the puncture, with value $\H(y)$. Hence, it is known that the graph $\Sigma$ extends smoothly to the ball $\Omega$, see e.g. \cite{LeRo}. In particular, the unit normal at the puncture is vertical. This is a contradiction with $|y|<1$.
\end{proof}
Thus, if an orbit converges to the axis $\{x=0\}$, it does to the points $(0,\pm 1)$. Recall that any such an orbit would generate an $\Hs$ intersecting orthogonally the axis of rotation. 

Recall that for any $(x_0,y_0)\in\Theta_\varepsilon^\kappa$, there exists an orbit passing through $(x_0,y_0)$ that is a solution of system \eqref{1ordersys}, as a consequence of Cauchy problem existence and uniqueness. However, Equation \eqref{1ordersys} has a singularity at the points with $x_0=0$, and thus we cannot guarantee the existence of an orbit having as endpoint $(0,\pm 1)$ by means of the Cauchy problem. To prove the existence of such an orbit we take advantage of the work of Gálvez and Mira \cite{GaMi2}, where the authors have studied the existence and symmetries of \emph{Weingarten spheres}\footnote{A \emph{Weingarten sphere} is a topological sphere whose mean curvature $H$, Gauss curvature $K$ and extrinsic curvature $K_e$ satisfy a smooth elliptic relation $\Phi(H,K,K_e)=0$} in homogeneous three-manifolds. Indeed, in Section 4.1, which has a strong interest in itself, they solved the Dirichlet problem for radial solutions of an arbitrary \emph{fully nonlinear elliptic PDE}. The following lemma is a straightforward consequence of the fact that our ODE \eqref{ODE2} is a particular case of this study.

\begin{lem}\label{canoex}
Let be $\hi$ and $\delta=\pm 1$. Then, there exists a disk $\Omega\subset\M^2$ containing the point $(0,0,1)$ and a function $u:\Omega\rightarrow\R$ such that the surface defined by $\sig=\mathrm{graph}(u)$ is an $\H$-surface in $\m2r$, which is rotationally symmetric around the vertical axis $\{(0,0,1)\}\times\R$ and that meets this axis in an orthogonal way at some $p\in\sig$, with unit normal at $p$ given by $\delta \partial_z$.

Moreover, $\sig$ is unique among all the graphical $\H$-surfaces over $\Omega$ with constant Dirichlet data.
\end{lem}

This lemma has the following implication in the phase plane $\Theta^\kappa_\varepsilon$.

\begin{cor}\label{ejefase}
Assume that $\H(\delta)\neq 0$ for $\delta \in \{-1,1\}$, and consider $\varepsilon\in \{-1,1\}$ such that $\varepsilon \H(\delta)>0$. Then, there exists a unique orbit in $\Theta^\kappa_{\varepsilon}$ that has $(0,\delta)\in \overline{\Theta^\kappa_{\varepsilon}}$ as an endpoint. There is no such an orbit in $\Theta^\kappa_{-\varepsilon}$.
\end{cor}
\begin{proof}
Let $\Sigma$ be the rotational $\H$-surface given for $\delta$ by Lemma \ref{canoex}. Let $\alpha_\kappa(s)=(x(s),z(s))$ be the profile curve of $\Sigma$, defined for $s\in [0,s_0)$ or $s\in (-s_0,0]$ depending on the orientation chosen on $\alpha_\kappa$, and assume that $\alpha_\kappa(0)$ corresponds to the point $p_0$ of orthogonal intersection of $\Sigma$ with its rotation axis. The mean curvature comparison theorem ensures us that all the principal curvatures of $\Sigma$ at $p_0$ have the same sign as $\H(\delta)$.

By \eqref{curvaprincipales} the geodesic curvature of $\alpha_\kappa(s)$ at $s=0$ is non-zero, and thus the sign of $z'(s)$ is constant for $s$ small enough. It follows then by \eqref{signk} that $\varepsilon \H(\delta)>0$. Consequently, the profile curve $\alpha_\kappa(s)$ generates an orbit in the phase plane $\Theta^\kappa_{\varepsilon}$ with $(0,\delta)$ as an endpoint. It is also clear from this argument that such an orbit cannot exist in $\Theta^\kappa_{-\varepsilon}$, because of the condition $\varepsilon \H(\delta)>0$.
\end{proof}

\subsection{Necessary conditions for the existence of $\H$-spheres}\label{necesariaesfera}
\vspace{-.3cm}

Once we have introduced the phase plane and analyzed the behavior of its solutions, we derive some necessary conditions for the existence of rotational $\H$-spheres. We emphasize again that for sphere we mean \emph{any} immersed (possibly with self-intersections), closed surface of genus zero.

The first result concerns the value of the mean curvature of a closed $\Hs$ immersed in $\m2r$, not necessarily rotational, in its points with largest and lowest height, and the implications that this fact has in the prescribed function $\H$.

\begin{pro}\label{m2rrmm1}
Let be $\hi$, and suppose that there exists a closed $\H$-surface $\mathcal{K}_\H$ in $\m2r$. Then,
\begin{itemize}
\item[1.] If $\mathcal{K}_\H\subset\h2r$, then $\H(-1)\H(1)>0$.
\item[2.] If $\mathcal{K}_\H\subset\s2r$, one of the following items holds:
\begin{itemize}
\item $\H(-1)\H(1)>0$.
\item $\H(-1)\H(1)=0$, and $\mathcal{K}_\H$ is a horizontal plane $\S^2\times\{t_0\}$, for some $t_0\in\R$.
\end{itemize}
\end{itemize}
\end{pro}

\begin{proof}
Let be $\mathcal{K}_\H$ a closed $\Hs$ and $\eta$ its unit normal vector field. Let be $p,q\in \mathcal{K}_\H$ the points of $\mathcal{K}_\H$ with lowest height and largest height, respectively, and consider the foliation of $\m2r$ by horizontal planes $\M^2_t=\M^2\times\{t\},\ t\in\R$. Notice that we can change the orientation of each element of this 1-parameter family without changing the value of the mean curvature, as it vanishes identically. Take some $\M^2_t$ and move it by vertical translations by decreasing $t$, until $\M^2_t\cap \mathcal{K}_\H=\varnothing$. Then, we move $\M^2_t$ towards $\mathcal{K}_\H$ by increasing $t$ until we reach at some instant $t_0$ a first contact point $p_0$ with $\mathcal{K}_\H$. This point is necessary an interior one, since both surfaces have no boundary. 

First, suppose that $\H(1)\H(-1)\neq 0$. Assume moreover that $\eta_p=\partial_z$, since the case $\eta_p=-\partial_z$ is proven similarly after a change of the orientation. As $\mathcal{K}_\H$ lies above $\M^2_{t_0}$ around $p$, the mean curvature comparison theorem ensures us that $0< H_{\mathcal{K}_\H}(p)=\H(\langle\eta_p,\partial_z\rangle)=\H(1)$. Now keep moving $\M^2_t$ upwards by increasing $t$ until we reach a last instant $t_1>t_0$ where $\mathcal{K}_\H$ and $\M^2_{t_1}$ intersect for the last time in a tangent point $p_1$. Note that if $\eta_q=\partial_z$, then we would have $H_{\mathcal{K}_\H}(q)=\H(1)> 0$, but this would contradict the mean curvature comparison principle since $\M^2_{t_1}$ is a minimal surface lying above $\mathcal{K}_\H$ around $q$. Thus, necessarily we have $\eta_q=-\partial_z$ and again the mean curvature comparison principle ensures us that $\H(-1)>0$, and the first item holds. 

Notice that we have proven implicitly that in a closed surface $\mathcal{K}_\H$, the unit normals $\eta_p$ and $\eta_q$ of the points $p$ and $q$ with largest and lowest height, respectively, are vertical and opposite. 

Now, suppose that $\H(-1)\H(1)=0$, and without losing generality assume that $\H(1)=0$. Then, either $\eta_p$ or $\eta_q$ is the vertical vector $\partial_z$, say $\eta_p$. In this situation, the horizontal plane $\M^2_{t_0}$ is tangent at $p$, where both unit normals agree. According with the maximum principle for $\Hss$, see Lemma \ref{pmax}, $\mathcal{K}_\H$ should agree with $\M^2_{t_0}$, and in particular $\H$ would vanish identically. In $\h2r$ this is a contradiction, since horizontal planes are not closed surfaces. In $\s2r$ this implies that $\mathcal{K}_\H$ agrees with $\S^2\times\{t_0\}$, which is a closed, minimal surface. This proves Proposition \ref{m2rrmm1}.
\end{proof}

Now, we derive a necessary condition on the prescribed function $\H$ for the existence of a rotational $\H$-sphere in the space $\h2r$. Notice that some hypothesis on $\H$ is needed, since in $\h2r$ there exists a sphere with constant mean curvature equal to $H_0$ if and only if $H_0>1/2$. The value $1/2$ is known as \emph{critical} and, in fact, it is optimal; for $H_0=1/2$ there exists a rotational, entire vertical graph in $\h2r$, incapacitating the existence of a sphere with constant mean curvature equal to $1/2$. The next proposition generalizes this necessary fact to the class of rotational $\Hss$.

\begin{pro}\label{condesferah2r}
Let be $\H\in C^1([-1,1])$ such that there exists a rotational $\H$-sphere $S_\H$ in $\h2r$. Then $2|\H(y)|>\sqrt{1-y^2}$ for every $y\in[-1,1]$. In particular, $\H$ never vanishes.
\end{pro}
\begin{proof}
Because $S_\H$ is closed, Proposition \ref{m2rrmm1} asserts that $\H(\pm 1)\neq 0$. We suppose that $\H(1)>0$, since the case $\H(1)<0$ is proven similarly.

Let $p,q$ the points of intersection of $S_\H$, with the axis of rotation, and suppose that $p_3<q_3$, where $x_3$ stands for the height of a point $x\in\m2r$. Let $\eta$ be the unit normal of $S_\H$. Then, it is clear that $\eta_p$ and $\eta_q$ are both vertical vectors, i.e. they point in the $\partial_z$ or $-\partial_z$ direction. By the mean curvature comparison principle and by Proposition \ref{m2rrmm1}, we have that $\eta_p=\partial_z$ and $\eta_q=-\partial_z$. In particular, we have that $\H(-1)>0$ also holds.

Suppose that the axis of rotation of $S_\H$ is the vertical line passing through the origin. Then, $S_\H$ is generated by the rotation of an arc-length parametrized curve as in Equation \eqref{parametrot}, and generates an orbit $\gamma(s)=(x(s),y(s))$ in $\Theta_1^{-1}$ having $(0,\pm 1)$ as endpoints, which correspond to the points $p,q$ of intersection of $S_\H$ with the axis of rotation.

Now, as $\H(1)>0$, at the point $y=1$ it is clear that the inequality $2\H(y)>\sqrt{1-y^2}$ holds. By continuity, for $y$ close enough to $y=1$ the function $\H(y)$ is positive. If the inequality $2\H(y)>\sqrt{1-y^2}$ fails to hold, let $y_0$ be the largest value in $(-1,1)$ such that $2\H(y_0)=\sqrt{1-y_0^2}$. Note that by continuity $\H(y_0)$ must be positive. Then, the horizontal graph $\Gamma_1^{-1}(y)$ given by \eqref{graga} has the point $(0,1)$ as endpoint and the line $\{y=y_0\}$ as an asymptote. Now, two possibilities can occur for $\Gamma_1^{-1}$ depending on the sign of $y_0$:
\begin{itemize}
\item The point $y_0$ satisfies $y_0\geq 0$. Then, the curve $\Gamma_1^{-1}$ is strictly contained in the half-strip $[0,\infty)\times (y_0,1]$. By properness and by the monotonicity properties in $\Theta_1^{-1}$, $\gamma(s)$ converges also to $\{y=y_0\}$, generating an entire, strictly convex graph and contradicting the compactness of $S_\H$.

\item The point $y_0$ satisfies $y_0<0$. Then,  the curve $\Gamma_1^{-1}$ intersects the axis $\{y=0\}$ and the equilibrium $e_0^{1,-1}$ exists, see Equation \eqref{equilibrium}. Again, by properness and the monotonicity properties of the phase plane $\Theta_1^{-1}$, the orbit $\gamma(s)$ must converge to $e_0^{1,-1}$, and thus the surface $S_\H$ should converge to a vertical cylinder, contradicting again the compactness of $S_\H$.
\end{itemize}
In any case, if $2\H(y)>\sqrt{1-y^2}$ fails to hold, we arrive to a contradiction.

This proves Proposition \ref{condesferah2r} for the case that $\H(1)>0$. If $\H(1)<0$, then we would arrive to $2\H(y)<-\sqrt{1-y^2}$; note that this condition is just $2\H(y)>\sqrt{1-y^2}$ after a change of the orientation in $S_\H$, and thus its proof is similar. This completes the proof of Proposition \ref{condesferah2r}.
\end{proof}

\begin{obs}
It is clear that for the particular choice $\H=H_0\in\R^+$, the condition  $2\H(y)>\sqrt{1-y^2},\ \forall y\in (-1,1)$ is just that $H_0$ has to be greater than $1/2$.
\end{obs}

In $\s2r$ we know that there exist rotational, compact, minimal surfaces; for instance, the horizontal planes $\S^2\times\{t_0\},\ t_0\in\R$ are surfaces satisfying these hypothesis. Thus, the condition $\H>0$ is no longer a necessary one for the existence of $\H$-spheres in $\s2r$. However, we give a necessary condition on the multiplicity of the zeros of $\H$. For that, recall that we suppose that $\H$ has finite zeros of finite multiplicity.

\begin{pro}\label{condesferas2r}
Let be $\hi$, and suppose that there exists an $\H$-sphere $S_\H$ in $\s2r$. Let be $z_1,...,z_n$ the zeros of $\H$, and denote by $m(z_i)$ the multiplicity of each $z_i,\ i=1,...,n$. Then, $\sum_{i=1}^n m(z_i)$ is even.
\end{pro}

\begin{proof}
As $S_\H$ is closed, by Proposition \ref{m2rrmm1} we can suppose that $\H(-1)\H(1)>0$. Indeed, if $\H(-1)$ or $\H(1)$ were equal to zero, then $S_\H$ would be a minimal horizontal plane $\S^2\times\{t_0\}$, for some $t_0\in\R$, contradicting that $\H$ has isolated zeros. After a change of the orientation, we suppose that $\H(-1)$ and $\H(1)$ are both positive.

Arguing by contradiction, suppose that $\sum_{i=1}^n m(z_i)$ is odd. Let us denote by $z_1,...,z_{k+1}$ the zeros of $\H$ with odd multiplicity and by $z_{k+2},...,z_n$ the zeros of $\H$ with even multiplicity. Since we suppose that $\sum_{i=1}^n m(z_i)$ is odd, then $k+1$ must be odd as well, and thus $k=2q$ for some $q\in\N$. 

By Equation \eqref{graga}, the curve $\Gamma_1^1$ in the phase plane $\Theta_1^1$ has the point $(0,1)$ as endpoint. By Equation \eqref{extensions2r}, we have that the curve $\Gamma_1^1$ must have the point $(\pi,-1)$ as its other endpoint. For that, we should note that $\H$ does not change it sign at the zeros $z_{k+2},...,z_n$ with even multiplicity, and in the zeros $z_1,...,z_{k+1}$ with odd multiplicity it changes it sign. Since there are an odd number of zeros with odd multiplicity, the last extension of $\Gamma_1^1$ is given by $\pi+\Gamma_1^1(y)$, which takes the value $\pi$ at $y=-1$.

Let $\gamma_0$ be the orbit in $\Theta_1^1$ associated to the $\H$-sphere $S_\H$. Let be $p,q\in S_\H$ the points of intersection of $S_\H$ with the axis of rotation, and such that $p_3<q_3$. On the one hand, the orbit $\gamma_0$ has its endpoints at the points $(0,1)$ and $(0,-1)$; these points correspond to the points $p$ and $q$, respectively. On the other hand, if we start at the point $(0,1)$, the monotonicity properties of the phase plane would yield that $\gamma_0$ should converge to the equilibrium $e_0^{1,1}$, contradicting the compactness of $S_\H$. In particular, $\gamma_0$ would never reach its endpoint $(0,-1)$. This contradiction proves Proposition \ref{condesferas2r}.
\end{proof}

The last result shows that the angle function of a rotational $\H$-sphere is a monotonous function. This fact will allow us to state that rotational $\H$-spheres are unique in the Hopf sense, according to Gálvez and Mira uniqueness theorem \cite{GaMi1}.

\begin{teo}\label{teoes}
Let be $\hi$ and suppose that $S_\H$ is a rotationally symmetric $\H$-sphere in $\m2r$. Then, the angle function of $S_\H$ is a monotonous function. 
\end{teo}

\begin{proof}
Because $S_\H$ is a sphere, in particular it is compact. Thus, Proposition \ref{m2rrmm1} ensures us that both $\H(1)$ and $\H(-1)$ have the same sign, which can be supposed to be positive after a change of the orientation.\footnote{We drop here the case that $S_\H=\S^2\times\{t_0\}$ in the space $\s2r$, since the result holds trivially.}

Let $\alpha_\kappa(s)$ be the profile curve of $S_\H$ given by Equation \eqref{parametrot}. By Propositions \ref{condesferah2r} and \ref{condesferas2r}, and because $\H(1),\H(-1)$ are both positive, the curve $\Gamma^{\kappa}_1$ is a compact connected arc with endpoints $(0,1)$ and $(0,-1)$. Hence, in both $\Theta_1^\kappa$ we have four monotonicity regions $\Lambda_1^\kappa,\dots,\Lambda_4^\kappa$ with monotonicities given by Lemma \ref{monotonia} and an equilibrium $e_0^{1,\kappa}$; see Figure \ref{figestre}. 

\begin{figure}[h]
\begin{center}
\includegraphics[width=.55\textwidth]{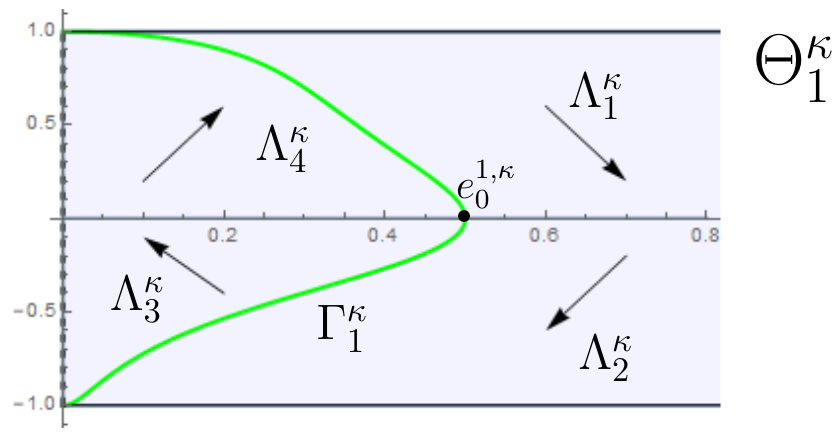}
\end{center}
\caption{The phase plane $\Theta_1^\kappa$, showing the monotonicity direction of each region $\Lambda_i^\kappa$.}
\label{figestre}
\end{figure}

By Corollary \ref{ejefase}, there exists an orbit $\gamma$ in $\Theta_1^\kappa$ that has $(0,1)$ as an endpoint. This orbit corresponds to an open subset of $S_\H$ that intersects the axis of rotation orthogonally with unit normal $\partial_z$. By the monotonicity properties, $\gamma$ stays in $\Lambda_1^\kappa$ for points near to $(0,1)$. Notice also that Proposition \ref{orbitapuntolimite} forbids the orbit $\gamma$ to have a point $(0,y)$ with $|y|<1$ as limit point. Indeed, in such an endpoint, the $\H$-sphere would be asymptotic to a vertical straight line, contradicting the compactness of $S_\H$, or would have a non-removable isolated singularity, which cannot happen because of the ellipticity of Equation \eqref{defHsup}. Also, since $\gamma$ cannot self-intersect (otherwise it would contradict the uniqueness of Cauchy problem), it is clear that $\gamma$ can behave in only two ways:

\begin{enumerate}
\item[i)] If $\gamma$ enters at some moment in the regions $\Lambda_3^\kappa$ or $\Lambda_4^\kappa$, then $\gamma$ has to converge asymptotically to the equilibrium $e_0^{1,\kappa}$ of $\Theta_1^\kappa$. But this implies that the profile curve $\alpha_\kappa(s)$ is asymptotic to a vertical straight line, i.e. $S_\H$ is asymptotic to a cylinder, contradicting the compactness of $S_\H$. Thus, this case is impossible.

\item[ii)] If $\gamma$ stays in $\Lambda_1^\kappa \cup \Lambda_2^\kappa$, then it is a graph of the form $x=g(y)>0$, with $y\in (y_0,1)$ for some $y_0\in [-1,1)$. By compactness of $S_\H$ we must have $y_0=-1$. Thus, $\gamma$ can be extended to a compact graph $x=g(y)\geq 0$ for $y\in [-1,1]$, and it has a second endpoint at some $(x_1,-1)$ with $x_1=g(-1)\geq 0$.
\end{enumerate}

Now we repeat the arguments above but starting at the point $(0,-1)$, obtaining an orbit $\sigma$ in $\Lambda_1^\kappa\cup \Lambda_2^\kappa\subset \Theta_1^\kappa$. We conclude that $\sigma$ can be extended to a graph $x=t(y)\geq 0$ for $y\in [-1,1]$, with a second endpoint at some $(x_2,1)$ with $x_2=t(1)\geq 0$. Since $\gamma$ and $\sigma$ cannot intersect on $\Theta_1^\kappa$, the only possibility is that $x_1=0$ or $x_2=0$. Thus, by the uniqueness property of Corollary \ref{ejefase}, we have $\gamma=\sigma$, which is then an orbit in $\Theta_1^\kappa$ joining $(0,1)$ with $(0,-1)$. Since, again by Corollary \ref{ejefase}, there are no orbits in $\Theta_{-1}^\kappa$ having any of such points as an endpoint, we conclude that $\gamma$ is the whole orbit that describes the profile curve $\alfa_\kappa(s)$. 

By Equation \eqref{1ordersys}, and since $\gamma$ stays in the region $\Lambda_1^\kappa\cup \Lambda_2^\kappa$, it follows that $y'(s)<0$ for all $s$. This implies that the angle function of $S_\H$, $\nu(\alpha_\kappa(s))=x'(s)$, is a monotonous function, completing the proof.
\end{proof}

\section{\large A Delaunay-type classification result}\label{clasification}
\vspace{-.5cm}

Given a positive constant $H_0$, a classical theorem due to Delaunay classifies, up to ambient isometries, the complete, rotational surfaces in $\r3$ with constant mean curvature $H_0$ as follows: the totally umbilical sphere $\S^2(1/H_0)$, the right circular cylinder $\S^1(1/(2H_0))\times\R$, a 1-parameter family of properly embedded \emph{unduloids}, and a 1-parameter family of properly immersed (non-embedded) \emph{nodoids}. Moreover, both the unduloids and the nodoids are invariant by the discrete group generated by some vertical translation in $\r3$.

This result has been generalized for CMC surfaces in further ambient spaces. Regarding the product spaces, we refer the reader to the papers \cite{HsHs,PeRi}. We must emphasize that in the space $\s2r$, Pedrosa and Ritoré also described the existence of a rotational, embedded torus of positive constant mean curvature.

Let us define the following set of functions:
\begin{equation}\label{espaciosc1h}
\mathfrak{C}^1_{\kappa}=\Big\{\H\in C^1([-1,1]);\ \H(y)=\H(-y), \forall y\in [-1,1],\ \mathrm{and}\ 2|\H(y)|>\sqrt{1-y^2},\ \mathrm{if}\ \kappa=-1\Big\}
\end{equation}
If $\kappa=1$, this set is just the set of $C^1$ even functions defined on $[-1,1]$. Note that every $\hipar$ lies in the hypothesis of either Proposition \ref{condesferah2r} or \ref{condesferas2r}, depending if $\kappa=-1$ or $\kappa=1$ respectively.

The aim of this section is to generalize Delaunay's theorem to the class of rotational $\Hss$, giving a similar description under the assumption that $\hipar$. We should point out that, in general, this classification result is no longer true for an arbitrary prescribed function $\H\in C^1([-1,1])$. For instance, if $2\H(y_0)=\sqrt{1-y_0^2}$ for some $y_0\in [-1,1]$ when $\kappa=-1$, an $\H$-sphere cannot exist by Proposition \ref{condesferah2r}. Also, for an arbitrary $\hi$ the statement in Proposition \ref{condesferas2r} in general does not hold, making impossible the existence of an $\H$-sphere in $\s2r$.

First we prove the existence of an $\H$-sphere, provided that $\hipar$. It is worth to mention that a more general existence result of immersed spheres in a simply connected, homogeneous three-manifold and whose mean, Gauss and extrinsic curvatures $H,K$ and $K_e$, respectively, satisfy a general Weingarten relation of the form $\Phi(H,K,K_e)=0$, has been recently obtained in \cite{GaMi2}. The improvement in this paper is that we present geometric necessary and sufficient conditions for the existence of prescribed mean curvature spheres.

\begin{teo}\label{existenciaesfera}
For each $\hipar$, there exists a rotational, embedded $\H$-sphere $S_\H$ in $\m2r$.
\end{teo}

\begin{proof}
The fact that $\H$ is even has the following consequence in the phase plane $\Theta_\varepsilon^\kappa$: if $(x(s),y(s))$ is a solution to \eqref{1ordersys}, then $(x(-s),-y(-s))$ is also a solution to \eqref{1ordersys}. Geometrically, this means that any orbit of the phase planes $\Theta^\kappa_{\varepsilon}$ is symmetric with respect to the axis $\{y=0\}$. 

If $\kappa=-1$, then after a change of the orientation we can suppose that $2\H(y)>\sqrt{1-y^2}$ holds, and in particular $\H(-1)=\H(1)$ are both positive. This implies that the curve $\Gamma_1^{-1}$, given by Equation \eqref{graga} for $\varepsilon=1$ and $\kappa=-1$, is a compact, connected arc in the phase plane $\Theta_1^{-1}$, with the points $(0,1)$ and $(0,-1)$ as endpoints, and it does not appear in the phase plane $\Theta_{-1}^{-1}$.

If $\kappa=1$ and we have that $\H(-1)=\H(1)=0$, then the surfaces $\S^2\times\{t_0\},\ t_0\in\R$ are rotational, embedded $\H$-spheres with either $\partial_z$ or $-\partial_z$ as unit normal, and the result holds trivially. Thus, if $\kappa=1$ we suppose that $\H(-1)=\H(1)\neq 0$. Again, after a change of the orientation we can suppose that both are positive. Because $\H$ is even, in particular the sum of the multiplicity of its zeros is even, and by Equation \ref{gragaenfunciondelaotra} we deduce that the curve $\Gamma_1^1$, given by Equation \eqref{graga} for $\varepsilon=1$ and $\kappa=1$,  is a compact, connected arc in the phase plane $\Theta_1^1$ with the points $(0,1)$ and $(0,-1)$ as endpoints.

These properties ensure us that the phase planes $\Theta_1^\kappa$, for $\kappa=\pm1$, are divided into four connected components, and an orbit behaves in each component as detailed in Lemma \ref{monotonia}; see Figure \ref{faseesfera}.

\begin{figure}[H]
\centering
\includegraphics[width=0.6\textwidth]{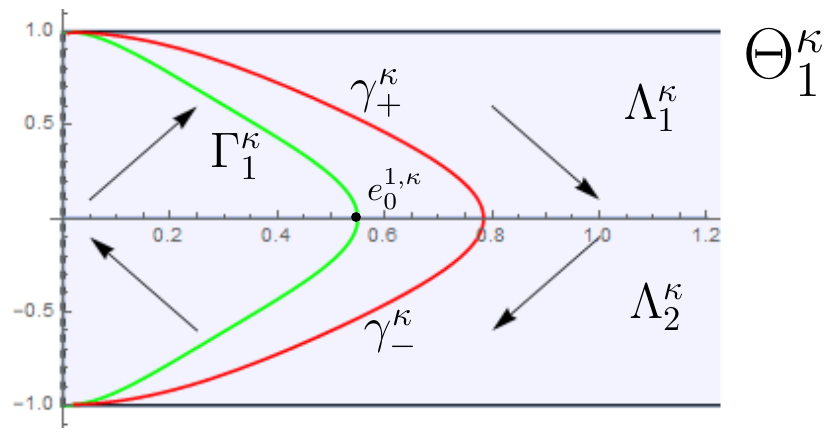}  
\caption{The phase plane $\Theta_1^\kappa$ for the function $\H(y)=1+y^2$, showing the equilibrium point $e_0^{1,\kappa}$, the monotonicity regions and their behaviors. The curve $\Gamma_1^\kappa$ is plotted in green, and the orbit $\gamma_+^\kappa\cup\gamma_-^\kappa$ corresponding to the $\H$-sphere plotted in red.}
\label{faseesfera}
\end{figure}

First, let $\sig^\kappa_+$ (resp. $\sig^\kappa_-$) be the $\H$-surface given by Lemma \ref{ejefase} intersecting orthogonally the axis of rotation and with unit normal equal to $\partial_z$ (resp. $-\partial_z$) at this intersection. We will denote by $\gamma^\kappa_+$ (resp. $\gamma^\kappa_-$) to the orbit in $\Theta_1^\kappa$ associated to $\sig^\kappa_+$ (resp. $\sig^\kappa_-$). Thus, $\gamma^\kappa_+$ is a curve in $\Theta^\kappa_+$ with $(0,1)$ as endpoint, that lies in $\Lambda^\kappa_1$ for points near to $(0,1)$. This also happens for $\gamma^\kappa_-$, which has $(0,-1)$ as endpoint and lies in $\Lambda^\kappa_2$ for points near to $(0,-1)$. By the symmetry condition and by uniqueness, if $\gamma^\kappa_+=(x(s),y(s))$ then $\gamma^\kappa_-=(x(-s),-y(-s))$. By the mean curvature comparison principle, the coordinate $x(s)$ of $\gamma_+^\kappa$ cannot diverge to infinity when the coordinate $y(s)$ approaches to some $y_0\geq 0$. Thus, $\gamma^\kappa_+$ has to converge to some finite point $(x_0,0),\ x_0>0,$ located at the axis $\{y=0\}$.

\textbf{Claim:} \emph{The point $(x_0,0)$  cannot be the equilibrium point $e_0^{1,\kappa}$.}

\emph{Proof of the claim:} Let us analyze the structure of the orbits of $\Theta^\kappa_1$ around $e_0^{1,\kappa}$. Because $\H$ is an even function, we have that $\H'(0)=0$. A straightforward computation yields that the linearized system at $e_0^{1,\kappa}$ associated to the nonlinear system \eqref{1ordersys} for $\varepsilon=1$ is
\begin{equation}\label{2ordersys}
\left(\begin{array}{c}
u\\
v
\end{array}\right)'=\left(\begin{array}{cc}
0 & 1\\
-\kappa-4\H(0)^2& 0
\end{array}\right) \left(\begin{array}{c}
u\\
v
\end{array}\right).
\end{equation}
The $a_{12}$ element of the linearized matrix $-\kappa-4\H(0)^2$ is always negative; for $\kappa=1$ is trivial, and for $\kappa=-1$ it follows from the hypothesis $2\H(y)>\sqrt{1-y^2}$ by just substituting at $y=0$. In this situation the orbits of Equation \eqref{2ordersys} are ellipses around the origin. By classical theory of nonlinear autonomous systems, this implies that there are two possible configurations for the space of orbits of \eqref{1ordersys} near $e_0^{1,\kappa}$; either all such orbits are closed curves (a \emph{center} structure), or they spiral around $e_0^{1,\kappa}$. However, this second possibility cannot happen, since all orbits of \eqref{1ordersys} are symmetric with respect to the axis $\{y=0\}$, and $e_0^{1,\kappa}$ belongs to this axis. In particular, we deduce that \emph{all orbits of $\Theta^\kappa_1$ stay at a positive distance from the equilibrium $e_0^{1,\kappa}$}. This proves the claim.

By properness, $\gamma^\kappa_+$ actually intersects the axis $\{y=0\}$ at some point $(x_0,0)$, $x_0>e_0^{1,\kappa}$, and can be expressed as a graph $\gamma_+^\kappa=(x,f(x))$, with $f(x)$ satisfying $f(0)=1,\ f(x_0)=0$ and $f'(x)<0$. By symmetry, the same holds for $\gamma^\kappa_-$ by just defining the function $(x,-f(x))$, see Figure \ref{faseesfera}.

By Equation \eqref{curvaprincipales} the principal curvatures of each $\sig^\kappa_{\pm}$ are positive everywhere. In particular $\sig^\kappa_+$ is a compact graph intersecting the axis of rotation and having the circumference $\S^1(x_0)\times\{a\}$, for some $a\in\R$, as boundary. In this boundary, its unit normal $\eta$ is horizontal and points inwards. By symmetry, $\sig^\kappa_-$ is just the graph $\sig^\kappa_+$ reflected with respect to a horizontal plane; the symmetry condition on $\H$ induces these reflections as isometries for the class of $\Hss$. In particular $\sig^\kappa_-$ is a compact graph which has as boundary the circumference $\S^1(x_0)\times\{b\}$, for some $b\in\R$, and the unit normal at this boundary is also horizontal and points inwards. 

After a vertical translation, both $\sig^\kappa_{\pm}$ are symmetric bi-graphs with respect some horizontal plane, and with their unit normals agreeing along their boundaries. By uniqueness of the Cauchy problem, we can smoothly glue together both $\H$-surfaces obtaining a compact $\H$-surface with genus 0 which is embedded, i.e. an embedded, rotationally symmetric $\H$-sphere $S_\H$. In Figure \ref{esfera} we can see an $\H$-sphere plotted in the space $\h2r$. Here, and henceforth, we use the Poincaré disk model of $\mathbb{H}^2$ when plotting $\H$-surfaces in $\h2r$.

In particular, the orbit $\gamma_0^\kappa:=\gamma_+^\kappa\cup\gamma_-^\kappa$ generated by $S_\H$ in $\Theta_1^\kappa,$ is a compact, symmetric arc with respect to the axis $\{y=0\}$, that lies entirely in $\Lambda_1^\kappa\cup\Lambda_2^\kappa$ and has $(0,\pm 1)$ as endpoints. This proves Theorem \ref{existenciaesfera}.
\end{proof}

\begin{figure}[H]
\centering
\includegraphics[width=0.5\textwidth]{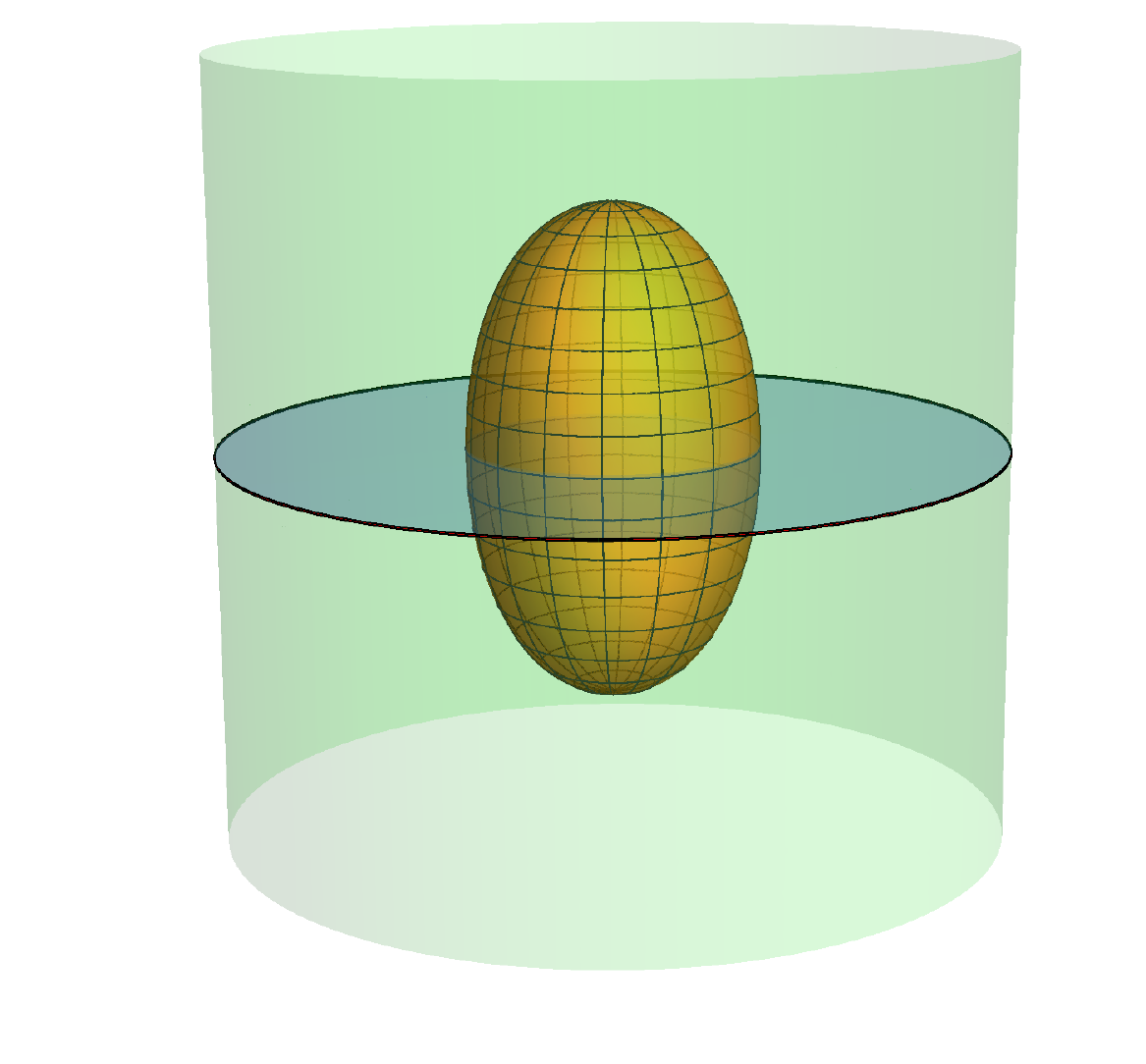}  
\caption{The rotational $\H$-sphere $S_\H$ in $\h2r$ for the function $\H(y)=1+y^2$. Note that $S_\H$ is a symmetric bi-graph over the horizontal plane $\mathbb{H}^2\times\{0\}$.}
\label{esfera}
\end{figure}

\begin{obs}
Let be $\hipar$ and $S_\H$ the corresponding rotational $\H$-sphere. In Theorem \ref{teoes} we proved that the angle function of $S_\H$ is strictly monotonous, and thus we can invoke Gálvez and Mira uniqueness Theorem to ensure that $S_\H$ is the only immersion of an $\H$-sphere in $\mkr$ (up to translations).
\end{obs}

Now we announce the Delaunay-type classification result for $\Hss$ in $\m2r$:

\begin{teo}\label{delaunay}
Let be $\hipar$. Then, up to isometries, the complete, rotational $\Hs$ in $\m2r$ are classified as follows:
\begin{enumerate}
\item There exists an $\H$-sphere $S_\H$.

\item There exists a vertical cylinder $C_\H$ of constant mean curvature $\H(0)$.

\item There exists a one parameter family of properly embedded $\H$-surfaces, $U_\H$, invariant by a vertical translation and the topology of an annulus, called $\H$-unduloids.

\item There exists a one parameter family of properly immersed $\H$-surfaces, $N_\H$, invariant by a vertical translation and the topology of an annulus, called $\H$-nodoids. 

\item In the space $\s2r$ there exist and embedded $\H$-surface diffeomorphic to $\S^1\times\S^1$, i.e. an embedded $\H$-torus.
\end{enumerate}
Moreover, both the $\H$-unduloids and the $\H$-nodoids are invariant by the discrete group generated by some vertical translation in $\m2r$.
\end{teo}

\begin{proof}
The existence of a rotational $\H$-sphere $S_\H$ was already proved in Theorem \ref{existenciaesfera}. In particular, we deduced the phase planes $\Theta_1^\kappa$ are divided into four monotonicity regions $\Lambda_i^\kappa,\ i=1,...,4,$ and that the orbit $\gamma_0^\kappa$ generated by $S_\H$ in $\Theta_1^\kappa$ is a compact arc, having $(0,\pm 1)$ as endpoints, symmetric with respect to the axis $\{y=0\}$ and that lies entirely in $\Lambda_1^\kappa\cup\Lambda_2^\kappa$; see Figure \ref{fasesk1}.

\begin{figure}[h]
\begin{center}
\includegraphics[width=.6\textwidth]{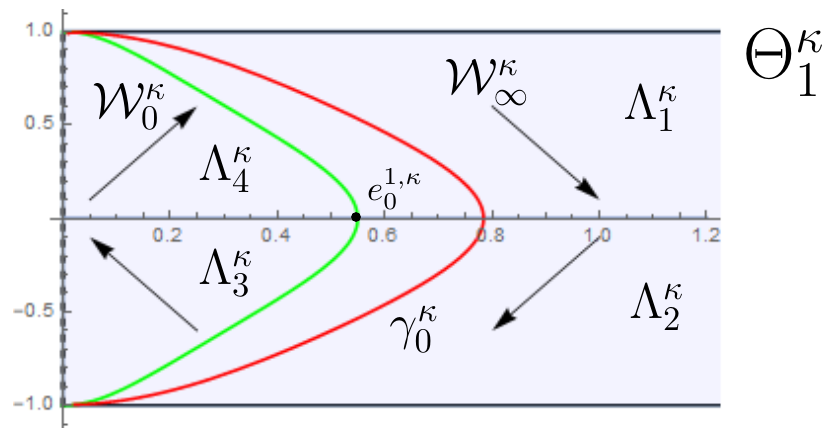}
\end{center}
\caption{The phase plane $\Theta_1^\kappa$ and the orbit $\gamma_0^\kappa$ corresponding to the $\H$-sphere.}
\label{fasesk1}
\end{figure}

Observe that there exists an equilibrium $e_0^{1,\kappa}\in \Theta^\kappa_1$, given by \eqref{equilibrium} if $\kappa=-1$ or by \eqref{equis2r} if $\kappa=1$, which generates a vertical cylinder with constant mean curvature equal to $\H(0)$. This proves Item $2$.

The orbit $\gamma_0^\kappa$ divides $\Theta_1^\kappa$ into two connected components: one containing the equilibrium $e_0^{1,\kappa}$, which we will denote by $\cW^\kappa_0$, and other denoted by $\cW^\kappa_\infty$; see again Figure \ref{fasesk1}.

If $\kappa=-1$, then $\cW^{-1}_\infty$ is unbounded; if $\kappa=1$, then $\cW^1_\infty$ is bounded, and the vertical segment $\{\pi\}\times(-1,1)$ belongs to its boundary (recall that this segment corresponds to the antipodal axis of rotation). Note that the uniqueness of the solution of the Cauchy problem guarantees that any orbit of $\Theta_1^\kappa$ lies entirely in one of these open sets.

Name $x_0^\kappa$ to the intersection of $\gamma_0^\kappa$ with the axis $\{y=0\}$, fix some $\xi^\kappa>x_0^\kappa$ and denote by $\gamma_1^\kappa$ to the orbit in $\Theta_1^\kappa$ passing through $\xi^\kappa$. By uniqueness of the Cauchy problem, it is clear that $\gamma^\kappa_1$ lies entirely in $\Lambda_1^\kappa\cup\Lambda_2^\kappa$. By properness, symmetry and monotonicity, $\gamma^\kappa_1$ can be expressed as a horizontal graph $x=g^\kappa(y)$ such that $g^\kappa$ is strictly increasing (resp. decreasing) in $(-1,0]$ (resp. in $[0,1)$), with $g^\kappa(0)=\xi^\kappa$ and $g^\kappa(\pm 1)=x_1^\kappa>0$, i.e. $\gamma^\kappa_1$ has the points $(x_1^\kappa,\pm 1)$ as endpoints, see Figure \ref{fasesnodoide}, left.

\begin{figure}[H]
\centering
\includegraphics[width=1\textwidth]{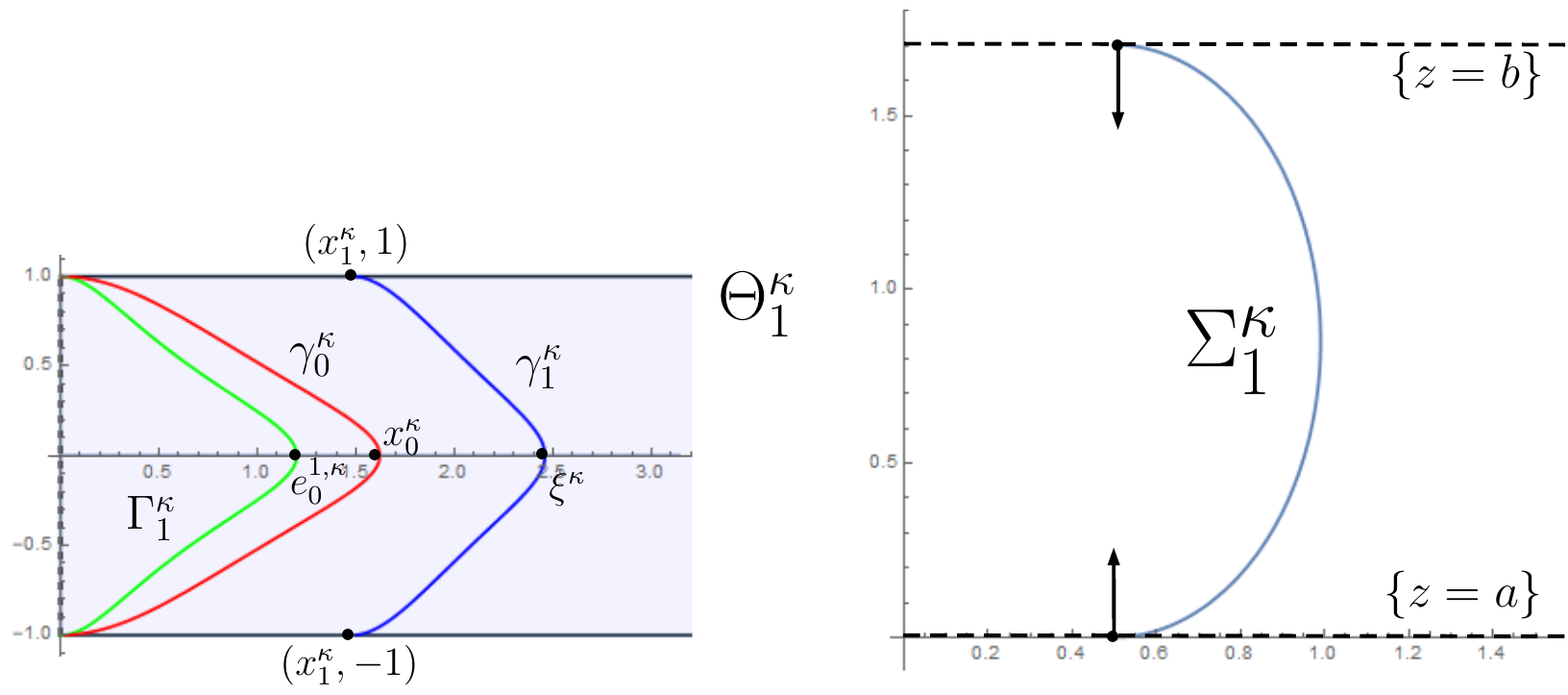}  
\caption{Left: the configuration of the phase plane $\Theta^\kappa_1$. Right: the profile curve corresponding to the orbit plotted in blue, which generates an $\Hs$ $\sig_1^\kappa$.}
\label{fasesnodoide}
\end{figure}

Let $\Sigma^\kappa_1$ denote the rotational $\H$-surface in $\m2r$ associated to any such orbit in $\cW^\kappa_{\infty}$, and let $\alfa_\kappa(s)=(x(s),z(s))$ be its profile curve. Note that $z'(s)>0$ since $\varepsilon=1$. Then, $\sig_1^\kappa$ is a compact, symmetric bi-graph over the domain $\Omega=\{x\in \M^2: x_1^\kappa\leq |x|\leq \xi^\kappa\}$, and its boundary is given by 
\begin{equation}
\parc \Sigma^\kappa_1 =( \S^1 (x_1^\kappa) \times \{a\} )\cup (\S^1 (x_1^\kappa) \times \{b\}),
\end{equation}
for some $a<b$. The $z(s)$-coordinate of the profile curve $\alfa_\kappa(s)$ of $\Sigma^\kappa_1$ is strictly increasing, and the unit normal to $\Sigma^\kappa_1$ along $\partial \Sigma^\kappa_1\cap \{z=a\}$ (resp. along $\parc \Sigma^\kappa_1\cap \{z=b\}$) is constant, and equal to $\partial_z$ (resp. to $-\partial_z$); see Figure \ref{fasesnodoide}, right.

Now we focus in the case $\varepsilon=-1$ and analyze the behavior of the orbits in the phase planes $\Theta_{-1}^\kappa$.

First, suppose that $\kappa=1$. Recall that by the symmetry of the phase planes w.r.t. the segment $\{x=\pi/2\}$, if $(x,y)$ denote the coordinates of the phase plane $\Theta_1^1$ then $\overline{(x,y)}:=(\pi-x,y)$ are the coordinates of the phase plane $\Theta_{-1,}^1$. In particular,  the curve $\Gamma^1_{-1}$ also exists in $\Theta_{-1}^1$, as well as the equilibrium $e_0^{-1,1}$, see Equations \eqref{gragaenfunciondelaotra} and \eqref{equis2r}. Bearing this in mind, if $\Lambda_1^1,...,\Lambda_4^1$ are the monotonicity regions of $\Theta_1^1$, then $\overline{\Lambda}_i^1=(\pi,0)-\Lambda_i,\ i=1,...,4$, are the monotonicity regions in $\Theta_{-1}^1$; see Figure \ref{fasesnodoide2}, right. Thus, the study of the phase plane $\Theta_{-1}^1$ reduces to the study of the phase plane $\Theta_1^1$.

\begin{figure}[H]
\centering
\includegraphics[width=1\textwidth]{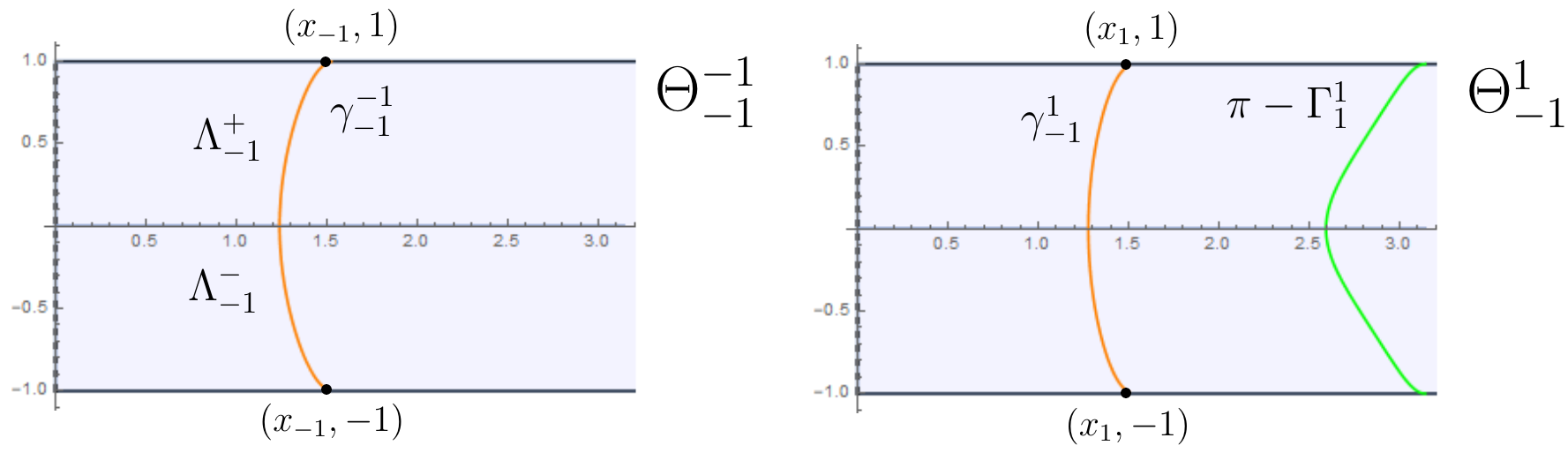}  
\caption{The orbits in the phase planes $\Theta_{-1}^\kappa,\ \kappa=\pm 1$.}
\label{fasesnodoide2}
\end{figure}

Suppose now that $\kappa=-1$. In this situation, the curve $\Gamma^{-1}_{-1}$ in $\Theta_{-1}^{-1}$ does not exist, and so $\Theta^{-1}_{-1}$ has only two monotonicity regions: $\Lambda_{-1}^+= \Theta^{-1}_{-1}\cap \{y>0\}$ and $\Lambda_{-1}^-=\Theta^{-1}_{-1}\cap \{y<0\}$, see Figure \ref{fasesnodoide2}, left. The description of the orbits in $\Theta_{-1}^{-1}$ follows easily from the monotonicity properties as explained in Lemma \ref{monotonia}. Any such orbit is given by a horizontal $C^1$ graph $x=g(y)$, with $g(y)=g(-y)>0$ for every $y\in (-1,1)$, and such that $g$ restricted to $[0,1)$ is strictly increasing. Note that the graph $g(y)$ cannot tend to $\infty$ as $y\to \pm 1$. On the contrary, the rotational $\H$-surface in $\h2r$ described by that orbit would be a symmetric bi-graph over the exterior of an open ball in $\mathbb{H}^2$. This is impossible by the maximum principle, since we would be able to compare with the $\H$-sphere $S_\H$. Thus, any orbit in $\Theta_{-1}^{-1}$ has as endpoints the points $(x_{-1},\pm1)$, where $x_{-1}=g(1)=g(-1)>0$.

Once we have analyzed both phase planes for $\varepsilon=-1$, we consider the orbit  $\gamma^\kappa_{-1}$ in $\Theta_{-1}^\kappa$ having as endpoints $(x_\kappa,\pm 1)$, and intersecting the axis $\{y=0\}$ at some $r_\kappa$, see Figure \ref{fasesnodoide2}. By similar arguments to the ones developed for $\Theta_1^\kappa$, we conclude that the generated $\Hs$ $\Sigma^\kappa_{-1}$ is a compact, symmetric bi-graph in $\m2r$ over some domain in $\M^2$ of the form $\{x\in \M^2: r_\kappa\leq |x|\leq x_\kappa \}$, and  
 \begin{equation}
\parc\Sigma_{-1}^\kappa=( \S^1(x_\kappa) \times \{c\} )\cup (\S^1(x_\kappa) \times \{d\}),
\end{equation}
for some $c<d$. This time, the unit normal of $\Sigma^\kappa_{-1}$ along $\parc \Sigma^\kappa_{-1}\cap \{z=c\}$ (resp. along $\parc \Sigma^\kappa_{-1}\cap \{z=d\}$) is $-\partial_z$ (resp. $\partial_z$). 

Consequently, by uniqueness of the solution to the Cauchy problem for $\H$-graphs in $\m2r$, we can deduce that given any $\xi^\kappa>x_0^\kappa>0$, the $\H$-surfaces $\Sigma^\kappa_{-1}$ and $\Sigma^\kappa_1$ as constructed above can be smoothly glued together along any of their boundary components where their unit normals agree, to form a larger $\H$-surface. For this, we should note that both $\Sigma^\kappa_{-1}$ and $\Sigma^\kappa_1$ are defined up to vertical translations in $\m2r$, and so we can assume without loss of generality in the previous construction that $a=d$ or that $b=c$ (and hence $\sig_1^\kappa$ and $\sig_{-1}^\kappa$ have the same Cauchy data). At this point, two possibilities may happen:

\begin{itemize}
\item[1.] We have simultaneously $a=d$ and $b=c$. In this situation, the surface obtained by gluing $\sig_1^\kappa$ with $\sig_{-1}^\kappa$ is an embedded $\Hs$ diffeomorphic to $\S^1\times\S^1$, that is an embedded $\H$-torus. If $\kappa=-1$, i.e. in the space $\h2r$, this is impossible in virtue of Proposition \ref{alexandrov}; see Figure \ref{nodoidetoro}, right.

\begin{figure}[H]
\centering
\includegraphics[width=1\textwidth]{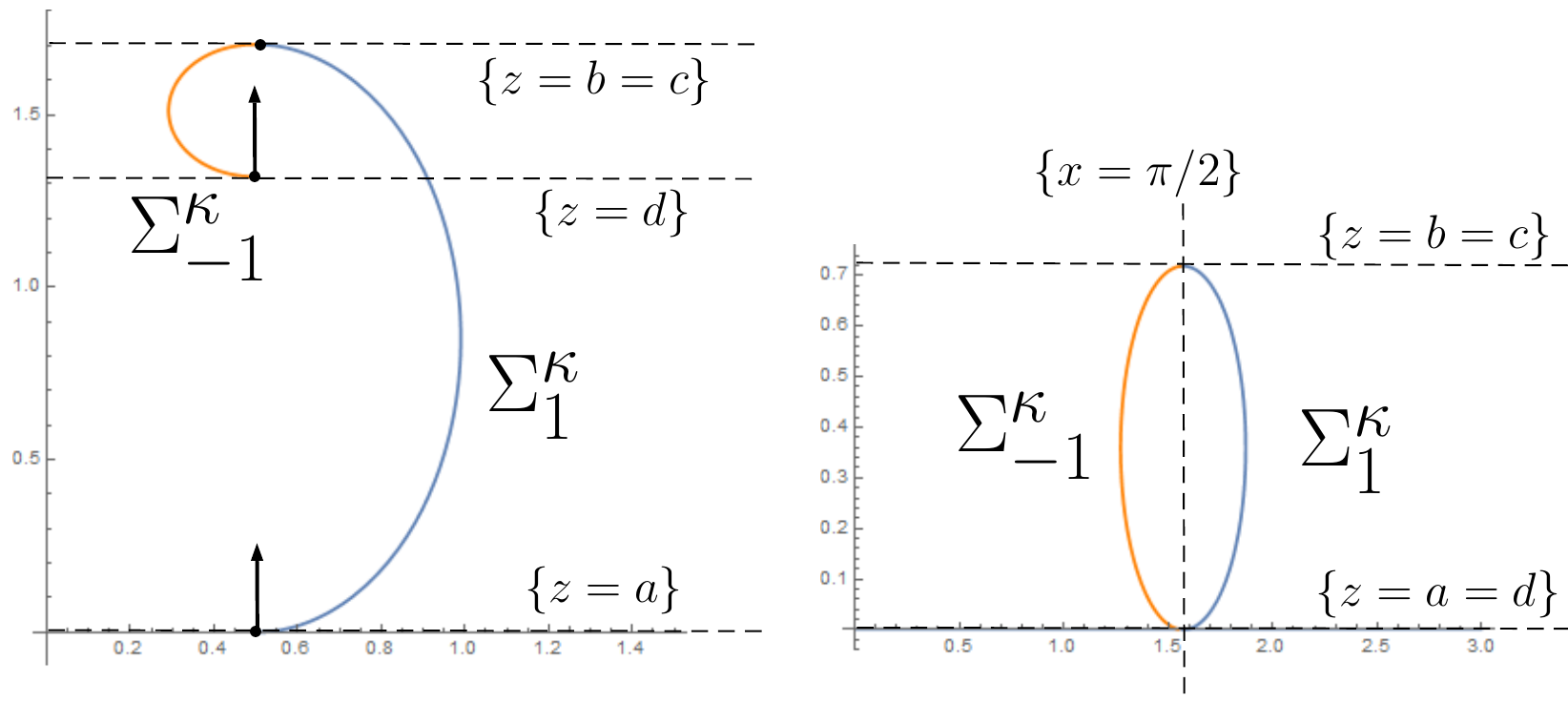}  
\caption{Left: the case $a\neq d$ and $b=c$, which generates an immersed nodoid in both $\m2r$. Right: the case $a=d$ and $b=c$, which generates an embedded torus in $\s2r$. Both profile curves have been obtained after a stereographic projection from $\M^2$ onto $\R^2$.}
\label{nodoidetoro}
\end{figure}

However, if $\kappa=1$ we know that the phase planes $\Theta_\varepsilon^1$ are symmetric w.r.t. the segment $\{x=\pi/2\}$, hence their orbits and their corresponding profile curves. By symmetry of the phase planes and uniqueness, we have $a=d$ and $b=c$ if and only if $x_1=\pi/2$. In this case, the orbit $\gamma_1^1$ in $\Theta_1^1$ passing through the points $(\pi/2,\pm 1)$ and the orbit $\gamma_{-1}^1$ in $\Theta_{-1}^1$ passing through the points $(\pi/2,\pm 1)$ are symmetric w.r.t. the segment $\{x=\pi/2\}$. This implies that the profile curves associated to $\gamma_1^1$ and $\gamma_{-1}^1$ are symmetric w.r.t. the plane $\{(x,y,0)\}\times\R,\ (x,y,0)\in\S^2$, see Figure \ref{nodoidetoro}, right, and thus the gluing of $\sig_1^1$ and $\sig_{-1}^1$ generates a rotational, embedded $\H$-torus in the space $\s2r$.

\item[2.] We have $a=d$ and $b\neq c$, or $a\neq d$ and $b=c$, see Figure \ref{nodoidetoro}, left. In that way, we iterate the previous process and obtain a proper, non-embedded rotational $\Hs$ diffeomorphic to $\S^1\times\R$ and invariant by some vertical translation, proving the existence of the $\H$-nodoids, see Figure \ref{nodoide}.
\end{itemize}

\begin{figure}[H]
\centering
\includegraphics[width=0.6\textwidth]{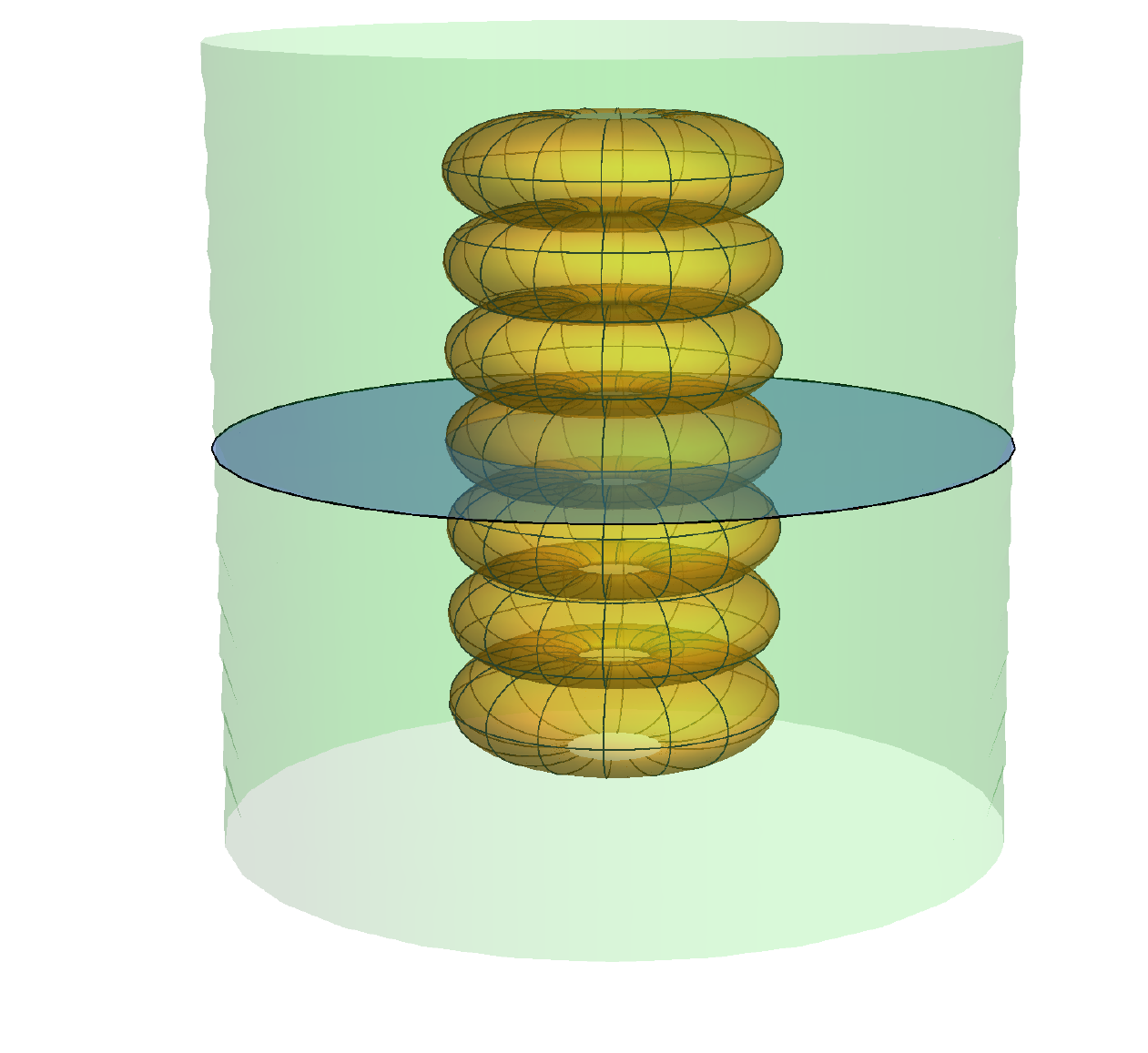}  
\caption{An $\H$-nodoid in $\h2r$ for the function $\H(y)=1+y^2$.}
\label{nodoide}
\end{figure}

Now, we consider an orbit $\gamma^\kappa$ of $\Theta^\kappa_1$ that is contained in the region $\cW^\kappa_0$. Recall that we pointed out in the proof of Theorem \ref{existenciaesfera} that every orbit stays at a positive distance from the equilibrium $e^{1,\kappa}_0$, and so $\gamma^\kappa$ does. As $\gamma^\kappa$ is symmetric with respect to the $\{y=0\}$ axis and taking into account the monotonocity properties of $\Theta^\kappa_1$, we see that only two possibilities can happen for $\gamma^\kappa$:

\begin{enumerate}
\item[1.] $\gamma^\kappa$ is a closed curve containing $e^{1,\kappa}_0$ in its inner region, or
\item[2.]  $\gamma^\kappa$ is a proper arc in $\Theta^\kappa_1$ with two limit endpoints of the form $(0,y_1)$, $(0,y_2)$, with $-1< y_1\leq 0\leq y_2<1$.
\end{enumerate}
However, according to Proposition \ref{orbitapuntolimite} no orbit can have a limit point of the form $(0,y)$ with $|y|<1$. Consequently, we deduce that any orbit $\gamma^\kappa$ inside $\cW^\kappa_0$ is a closed curve that contains $e^{1,\kappa}_0$ inside its inner region, see Figure \ref{onduloidefases}.

\begin{figure}[H]
\centering
\includegraphics[width=0.6\textwidth]{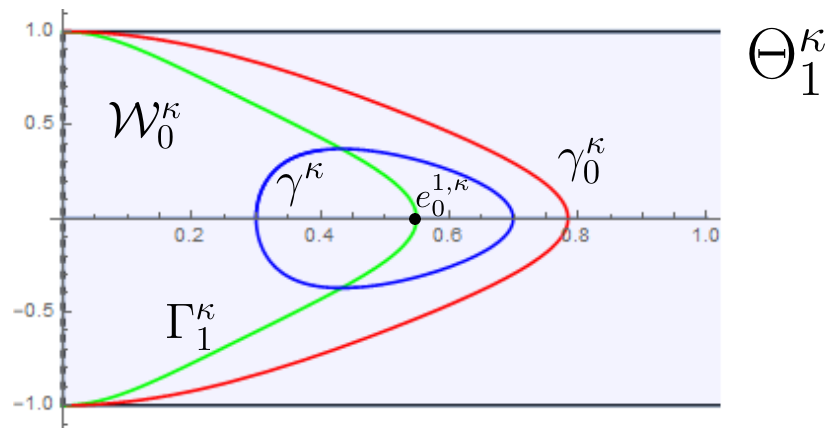}  
\caption{The phase plane $\Theta_1^\kappa$ and an orbit corresponding to an $\H$-unduloid.}
\label{onduloidefases}
\end{figure}

This implies that the profile curve $\alfa_\kappa(s)$ of the rotational $\H$-surface $\sig^\kappa$ associated to any such orbit satisfies that $z'(s)>0$ for all $s$ and that $x(s)$ is periodic. These properties imply that $\Sigma^\kappa$ is an $\H$-unduloid, with all the properties asserted in the statement of the theorem (see Figure \ref{onduloide}). 

This concludes the proof of Theorem \ref{delaunay}.
\end{proof}

\begin{figure}[H]
\centering
\includegraphics[width=0.5\textwidth]{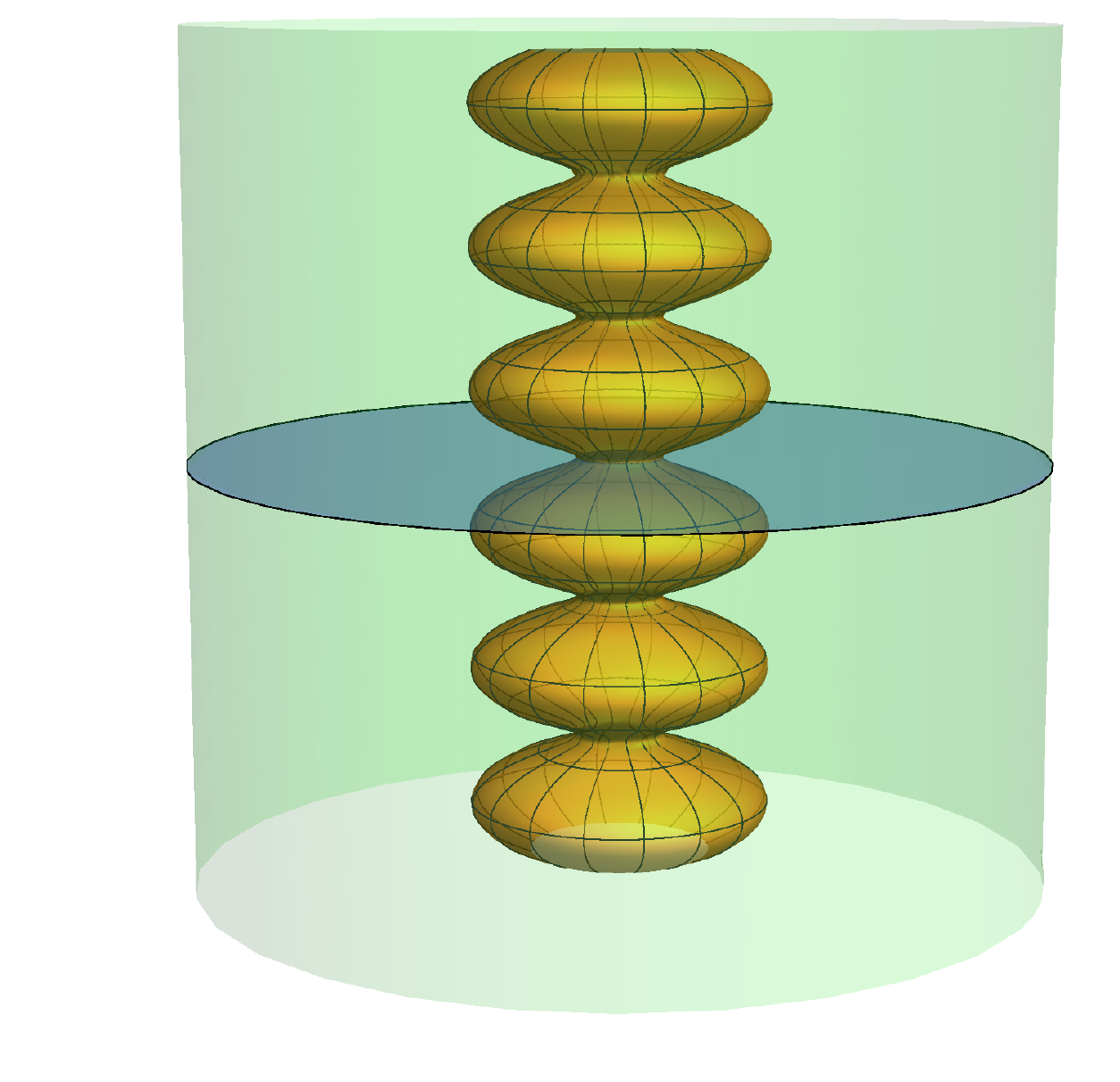}  
\caption{An $\H$-unduloid in $\h2r$ for the function $\H(y)=1+y^2$.}
\label{onduloide}
\end{figure}

\begin{remark}
Similarly to what happens in the CMC case and for $\H$-hypersurfaces in $\R^n$ \cite{BGM2}, the family of $\H$-unduloids is a continuous $1$-parameter family; at one extreme of the parameter, they converge to a (singular) vertical chain of tangent rotational $\H$-spheres $S_{\H}$; at the other extreme they converge to the CMC cylinder $C_{\H}$.
\end{remark}

\def\refname{References}

\end{document}